\documentclass[11pt,reqno]{amsart}
\usepackage{graphicx}
\pdfoutput=1
\usepackage{amsfonts,amsmath,amssymb}
\usepackage{hyperref}
\usepackage{paralist}
\usepackage[linesnumbered,ruled,vlined,norelsize]{algorithm2e}
\usepackage{stmaryrd}
\usepackage[usenames]{xcolor}
\newcommand{\cx}{{\mathbb{C}}}

\newcommand{\fps}[1]{\C\llbracket #1 \rrbracket}

\usepackage{thmtools}
\declaretheoremstyle[bodyfont=\normalfont]{noncursive}
\declaretheorem{theorem}
\declaretheorem[numberwithin=section]{lemma}

\declaretheorem[numberlike=lemma]{proposition}
\declaretheorem[numberlike=lemma]{corollary}
\declaretheorem[style=noncursive,numberlike=lemma]{definition}

\declaretheorem[style=noncursive,numberlike=lemma]{remark}

\newcommand{\im}{\ensuremath{\mbox{\rm Im}\,}}

\def\1#1{\overline{#1}}
\def\2#1{\widetilde{#1}}
\def\3#1{\widehat{#1}}
\def\4#1{\mathbb{#1}}
\def\5#1{\frak{#1}}
\def\6#1{{\mathcal{#1}}}

\newcommand{\CC}[1]{\mathbb{C}^{#1}}

\newcommand{\RR}[1]{\mathbb{R}^{#1}}

\newcommand{\lr}{\longrightarrow}

\numberwithin{equation}{section}

\newcommand{\C}{\mathbb{C}}

\newcommand{\R}{\mathbb{R}}

\newcommand{\N}{\mathbb{N}}

\newcommand{\dopt}[2]{\frac{\partial #1}{\partial #2}}

\def\label#1{\label{#1}{\sf (#1)}~}

\title[Equivalence of Cauchy-Riemann manifolds]{Equivalence of three\,--\,dimensional Cauchy-Riemann manifolds and multisummability theory}

\author {I. Kossovskiy}\thanks{The research of I.~Kossovskiy was supported by the Austrian Science Fund (FWF) and the Czech Grant Agency (GACR) grant 17-19437S}
\address{\parbox{0.8\linewidth}{%
        Department of Mathematics and Statistics, Masaryk University, Brno/\\ %
        Department of Mathematics, University of Vienna}
    }

\email{kossovskiyi@math.muni.cz, kossovi3@univie.ac.at}

\author {B. Lamel}\thanks{The research of B.~Lamel was supported by the Austrian Science Fund (FWF)}
\address{Department of Mathematics, University of Vienna}
\email{bernhard.lamel@univie.ac.at}

\author {L. Stolovitch}\thanks{The research of L. Stolovitch was supported by ANR-FWF grant "ANR-14-CE34-0002-01" for the project ``Dynamics and CR geometry'' and by ANR grant ``ANR-15-CE40-0001-03'' for the project ``BEKAM''}
\address{CNRS and Universit\'e C\^{o}te d'Azur, CNRS, LJAD,France.}
\email{stolo@unice.fr}


\begin{document}
\date{\today}
\begin{abstract}
We apply
 the {\em multisummability theory} from Dynamical Systems to CR-geometry. As the main result, we show that two real-analytic hypersurfaces in $\mathbb C^2$ are formally  equivalent, if 
 and only if they are $C^\infty$ CR-equivalent at the respective point.  As a corollary, we prove that all formal equivalences between {\em real-algebraic} Levi-nonflat hypersurfaces in $\mathbb C^2$  are algebraic (and in particular convergent). By doing so, we solve a Conjecture due to N.\,Mir \cite{problist}.  
\end{abstract}
\maketitle

\date{\today}
\tableofcontents

\section{Introduction}

In his 1907 paper \cite{poincare}, H.\,Poincar\'e made the fundamental discovery that 
real-analytic hypersurfaces in complex Euclidean spaces possess nontrivial
local invariants under the action of the pseudogroup of (local) biholomorphisms. Poincar\'e looked at local biholomorphisms as {\em complex power series maps}
\begin{equation}\label{formalp}
(z,w)\lr \left(z+\sum_{k+l\geq 2} a_{kl}z^kw^l,w+\sum_{k+l\geq 2} b_{kl}z^kw^l\right), \quad (z,w)\in\CC{2},
\end{equation}
 which in a natural sense act on real-analytic hypersurfaces. (Poincar\'e was not concerned with   convergence of such maps, that is, he dealt with {\em formal} biholomorphic transformations of hypersurfaces). He then showed that the action of formal biholomorphisms on the space of $k$-jets of defining functions of hypersurfaces is not transitive for sufficiently large $k$, and it implies the existence of the above biholomorphic invariants. This work of Poincar\'e is often considered as the starting point for studying the holomorphic geometry of real submanifolds in complex space, which (in an a bit more general setting) is referred to as   {\em Cauchy-Riemann geometry} (shortly: CR-geometry). 
 
 The cited work of Poincar\'e led to numerous developments in the subject of CR-geometry and raised a large number of interesting problems, some of which are open till present. The main result of this paper completes one of those: precisely, it solves the long standing problem of constructing a {\em geometric} realization for {\em formal} biholomorphic transformations \eqref{formalp} between { arbitrary} real-analytic  hypersurfaces in $\CC{2}$. The main tool which made  such a realization of formal maps of {\em arbitrary} real-analytic hypersurfaces possible is the modern theory of summability for formal power series transformations of dynamical systems, developed in the work of   Ramis, Sibuya, Ecale, Malgrange,  Braaksma, and 
Balser in the 1990's and referred to as the {\em multisummability theory}. 
 
 Before stating our main result in detail, we shall briefly outline  basics of CR-geometry. 
A hypersurface $M\subset\C^N$ 
gets endowed with a Cauchy-Riemann (CR) structure 
by its complex tangent bundle $T^c M$, whose fibers
are the  complex tangent planes
$T_p^c M = T_p M \cap i T_p M$, $p\in M$. 
This is a complex vector bundle over $M$
whose structure operator $J \colon T^c M \to T^c M $ is just multiplication 
by $i$ in $\C^N$.  The CR 
structure bundle is the subbundle $\mathcal{V} \subset \C TM$
whose fibers $\mathcal{V}_p$ consist of vectors of the form
$X_p + i J X_p$ with $X_p \in T_p^c M$. A function $f$ is 
said to satisfy the tangential Cauchy-Riemann equation 
or simply said to be CR if $\bar L f =0$ for every section 
$\bar L$ of $\mathcal{V}$. Holomorphic functions in a neighborhood of a manifold give basic examples of CR-functions. For
more on these, we refer the reader to 
\cite{ber}.

  Given 
two hypersurfaces $M , M^* \subset \C^N$, we 
say that a 
smooth map $H\colon M \to M^* $ is CR if 
the natural extension of its 
 differential $dH \colon \C TM \to \C T(M^*)$ 
restricts to the CR-structure bundle: 
 $dH \colon  \mathcal{V} \to \mathcal{V}^*$. It turns out that 
 this first-order system of PDEs on $M$ is equivalent to 
 requiring that if $H = (H_1 , \dots , H_N)$, then the components
 $H_j$ satisfy the tangential Cauchy-Riemann 
 equations, i.e. $\bar L H_j = 0$ for every section $\bar L$ of 
 $\mathcal{V}$ and $j=1,\dots, N$. Biholomorphisms of the ambient space transforming two hypersurfaces into each other are basic examples of CR-maps. 

Let us, for the rest of this paper, {\em consider only $\mathcal{C}^\infty$ smooth CR-functions and CR-maps}.
 Then, after a choice of holomorphic 
 coordinates $Z \in \C^N$, the Taylor series $T_p f$ of a 
 CR-function $f$ at a point $p$ can be identified with 
 a formal power series $T_p f \in \fps{z-p}^N$ (see e.g. \cite{dSL13,ber}),
 and therefore a CR-map $H\colon M \to M^*$ gives rise to 
 a formal power series map 
 $T_p H = (T_p H_1, \dots , T_p H_N) \in \fps{Z-p}^N $
 for every $p\in M$. (Here and below $\fps{z-p}^N$ denotes the ring of formal power series centered at $p$).

  On the
 other hand, if $M = \{ \varrho = 0 \}$ and $M^* = \left\{ \varrho^* =0 \right\}$ are given as the vanishing sets of 
 real-analytic defining functions $\varrho, \varrho^*$, and $p\in M$,  
 we define 
  a formal CR-map $\hat H\colon (M,p) \to M^*$ as a formal 
  power series map $\hat H \in \fps{Z-p}^N$  which in addition satisfies 
  the  
 formal condition 
 $\varrho^* \circ \hat H = A \varrho $ for some  $ A \in \fps{Z-p, \overline{Z-p}}$. 
Formal CR-maps $\hat H \colon 
(M,p) \to M^*$ whose Jacobian matrix is invertible therefore 
encode the above mentioned formal obstructions 
 to finding a smooth CR-diffeomorphism $H$ between the real hypersurfaces $M\subset \C^N$
 and $M^*\subset \C^N$ satisfying
  $H(p) = \hat H(p)$.  
We are going to say that $M$ and $M^*$ are {\em formally equivalent}
 at $p\in M$ and $p^* = H(p) \in M^*$
 if there exists an invertible 
 formal CR map $\hat H \colon (M,p) \to M^*$, and that they are CR equivalent if there exists a CR 
 diffeomorphism $H\colon (M,p)
  \to M^*$. In short our discussion up to 
 this point can be summarized as 
 follows: if $M$ and $M^*$ are CR
 equivalent, they are also formally equivalent 
 (at every point).

\smallskip

 We therefore have the following 
 natural problem: {\em does every formal CR-diffeomorphism $\hat H\colon (M,p) \to (M^*,p^*)$ arise 
 as the Taylor series $T_p H$ of a smooth CR map $H \colon M \to M^*$?  Or more generally:
 if $M$ and $M^*$ are formally 
 equivalent at 
 some point, are they CR equivalent?  } 
 
 \smallskip
 
 Note that both questions make sense to ask even in the category of merely {\em smooth} hypersurfaces. However, the answer is then trivially negative: for example, the flat at the origin perturbation $M=\{\im w=|z|^2+e^{-1/|z|^2}\}$ of the quadric $Q=\{\im w=|z|^2\}$  is formally equivalent to the quadric at the origin (their Taylor series simply coincide), however, it is not difficult to compute that $M$ has a generically non-vanishing CR-curvature (e.g. \cite{chern}) and hence is not CR-diffeomorphic to $Q$.      
 
 \smallskip 
 
 Our main theorem answers both 
  questions under discussion (in the real-analytic category) in the affirmative in 
 $\C^2$. 
 In order to state it, we also need to recall the 
 notion of being Levi-flat: we 
 say that a 
 hypersurface $M\subset \C^2$
 is Levi-flat if it is foliated by
 complex curves. Equivalently, 
 we can either require that the 
 distribution $T^c M \subset TM$
 is integrable. If 
 $M$ is real-analytic, another
 equivalent condition is that 
 in suitable local holomorphic
 coordinates $(z,w)$ at $p$, 
 $M$ can (locally) 
 be written as $\{\im w = 0\}$.

\begin{theorem}\label{main}
Let $M ,M^* \subset\CC{2}$ be two real-analytic hypersurfaces. Assume that $M$ and $M^*$  are formally equivalent at their reference points $p\in M,\,p^*\in M^*$. Then  $M$ and $M^*$ are  ($C^\infty$) CR-equivalent at the respective points $p,p^*$.
 If $M,M^*$ are in addition Levi-nonflat, then the given formal equivalence $\widehat H$ between them can be realized by a ($C^\infty$) CR-diffeomorphism $H:\,M\to M^*,\,H(p)=p^*$, whose Taylor series 
 at $p$ is $\widehat H$.
\end{theorem}

We shall particularly emphasize here that the proof of 
\autoref{main} exhibits an 
exciting and surprising 
application to CR-geometry of the {\em multisummability 
theory} from Dynamical Systems mentioned above (see Section 2.6 for details and references here). 

 We shall further note that the regularity asserted in \autoref{main} cannot be improved further! This 
is discussed below.

Finally, we shall explain that in the Levi-flat case a smooth realization of an arbitrary formal CR-diffeomorphism is not possible: for example, any map $z\mapsto f(z),\,w\mapsto w,\,  f(0)=0,\,f'(0)\neq 0$ with a formal and divergent $f(z)$ is a formal CR-automorphism of the Levi-flat model  $\{\im w = 0\}$, which can not be realized by a smooth CR-automorphism due to the divergence of $f(z)$. 
\smallskip

The content of 
\autoref{main} is the endpoint
(in $\C^2$) 
of a long development. The 
starting point was the 
case of Levi-nondegenerate
hypersurfaces 
$M,M^* \subset \C^2$. We 
recall that Levi-nondegeneracy
refers to 
the lowest order obstruction 
to being Levi-nonflat, i.e. $T^c M$ being nonintegrable: one says
that $M$ is Levi-nondegenerate 
at $p$ if 
for any two sections 
$X, Y$ of  $T^c M$ with 
$X_p$ and  $Y_p$ (real) linearly 
independent in $T^c M$,  the commutator $[X , Y]_p \notin T^c_p M$.  Levi-nondegenerate hypersurfaces in $\C^2$ have been classified 
up to local holomorphic equivalence, i.e. up to 
 agreeing in 
suitable local holomorphic coordinates, 
by E. Cartan \cite{cartan} (this
was later generalized  for 
Levi-nondegenerate hypersurfaces in $\C^N$,  $N>2$, by Tanaka~\cite{tanaka} and Chern-Moser~\cite{chern}). It turns out that 
the only obstruction to holomorphic equivalence of two 
Levi-nondegenerate  
hypersurfaces $M$ and $M^*$ in
$\C^2$ is that they are formally 
equivalent, and that 
in fact every 
formal CR-equivalence
between
Levi-nondegenerate hypersurfaces 
converges. 

Later on, it was shown 
by Baouendi, Ebenfelt, and Rothschild \cite{ber1} that this convergence 
phenomenon persists in the 
the case where the hypersurfaces are of finite type. Here one says
that a hypersurface $M\subset \C^2$ is 
of finite type at $p$ if for 
for any two sections 
$X, Y$ of  $T^c M$ with 
$X_p$ and  $Y_p$ (real) linearly 
independent in $T^c M$,  some commutator $[X , [X, \dots, [X,Y] \dots ]]_p \notin T^c_p M$.

The situation is quite different 
in the Levi-flat case: here, 
there are plenty of 
non-convergent formal maps, as 
any map of the form 
$H(z,w) = (f(z,w) , g(w)) \in \fps{z,w}^2$ 
which 
satisfies 
$g(w) = \overline{g} (w)$ maps
 $\im w = 0$ into itself. 

 The question what actually
 happens in the 
 case of a Levi-nonflat, but 
 {\em infinite type} hypersurface,
 has long remained open. We say
 that a hypersurface  $M\subset \C^2$ is of 
 infinite type at the 
 point $p$ if 
 the distribution $T^c M$ has 
 an integral submanifold 
 through the point $p$ (see e.g. \cite{ber}). It turns
 out that if $M$ is real-analytic,
 then
 any such integral manifold 
 necessarily is a nonsingular 
 complex curve (contained in $M$). This complex curve is 
 called the complex locus of $(M,p)$.

Contrary to the finite 
type case, 
in the infinite type case, 
there can be {\em divergent} formal 
CR maps. To be exact, the first 
author and Shafikov \cite{divergence} showed that 
{\em  there 
exist Levi-nonflat hypersurfaces 
$M, M^*  \subset \C^2$ with $M$ of 
infinite type at $p$ such that 
$M$ and $M^*$ are formally 
equivalent at $p$, and
such that every 
formal 
equivalence $\hat H \colon (M,p) \to M^*$ diverges.} 
On the 
other hand, the first and 
the second author showed 
\cite{nonanalytic} that { \em there
exist hypersurfaces $M,M^* \subset \C^2$, and a point $p \in M$
of infinite type, such that 
$M$ and $M^*$ are CR equivalent, 
but not holomorphically equivalent.}
These phenomena came as a surprise
to many specialists, and 
made the question of whether 
formal equivalence implies CR
equivalence an interesting one. 

We therefore see that the assertion of 
\autoref{main} cannot be further strengthened. In fact, \autoref{main} gives
 a complete answer to 
 the nature of 
 obstructions to CR equivalence
 of real-analytic 
 hypersurfaces in $\C^2$: 
 they are all purely formal. 

 Let us again 
 emphasize that even though
 we are dealing with 
 real-analytic hypersurfaces, 
 we establish the 
 existence of merely {\em smooth} CR-diffeomorphisms whose Taylor
 series agree with 
 a given formal CR equivalence. 
 The results of \cite{divergence,nonanalytic} mentioned above 
 actually show that 
 this is the best we can hope for. 
 We also emphasize that our result does not require any assumptions besides analyticity of the manifolds. It contrasts with similar results in Dynamical Systems such as Chen-Sternberg theorem \cite{stern2} concerning smooth classification of germs of vector fields at a fixed point, which require some {\it hyperbolicity } assumptions on the linear part of the vector field at the fixed point.

A nice application 
of \autoref{main} was observed 
by Nordine Mir, resolving a
a conjecture (by Mir) stated in 
 \cite{problist}. Before 
 we state this result, let 
 us recall that 
 a power series map
  $\hat H \in \fps{z-p}^N $ is
  algebraic if there
  exist nontrivial polynomials 
  $p_j (z,w)$ such that $p_j(z,\hat H_j(z)) = 0$ for every $j=1,\dots , N$. Every algebraic power
  series map is actually 
  convergent (one can see this from Artin's Approximation Theorem \cite{artin-approx}, see e.g. \cite{rond-survey}). Let us also recall 
  that a real-analytic
  hypersurface $M\subset \C^N$ is said 
  to be real-algebraic if it is  contained in
  the vanishing locus of 
  a nontrivial real polynomial in the 
  underlying real variables 
  in $\R^{2N} = \C^N$.

\begin{theorem}\label{mir}
Let $M,M^*\subset\CC{2}$ be two real-algebraic Levi-nonflat hypersurfaces, and $\widehat H:\,(M,p)\mapsto (M^*,p^*)$ a formal invertible CR-map. Then $\widehat H$ is algebraic, and in particular, $\widehat H$ is convergent. 
\end{theorem}

\begin{proof}[Proof of \autoref{mir}]
By \autoref{main}, 
we can find a CR-diffeomorphism 
 $H\colon (M,p) \to (M^*,p^*)$,  whose Taylor expansion  at $p$ coincides with $\widehat H$.  We 
 can now apply the algebraicity 
 theorem of 
  Baouendi, Huang and Rothschild \cite{bhr} (see also Webster \cite{webster}) to 
 conclude that  
 $H$ is an algebraic map (in particular, it is holomorphic). Since, again, $\widehat H$ is the Taylor expansion of $H$ at $p$, this implies the assertion of the theorem.
\end{proof}

As follows from the above mentioned theorem of Shafikov and the first author, the convergence phenomenon in \autoref{mir} is a specific feature of {\em algebraic} (but not general analytic!) hypersurfaces, similarly to the theorem of Baouendi-Huang-Rothschild. 

Before we describe our approach
to the problem 
in more detail, we will state 
a more technical version 
of \autoref{main} which contains 
considerable additional 
information. We recall that a  formal power series  $$\widehat H(z,w)=\sum_{k,l\geq 0}c_{kl}z^kw^l$$ (centered at the origin) is said to be {\em of the $(r,s)$ multi Gevrey class, $r,s>0$}, if there exist appropriate constants $A,B,C>0$ such that the Taylor coefficients $c_{kl},\,k,l\geq 0$  satisfy the bounds:
\begin{equation}\label{sbounds}
|c_{kl}|\leq A\cdot B^k\cdot C^l\cdot (k!)^r(l!)^s.
\end{equation}
 For the more technical concept of {\em Gevrey asymptotic expansion} we refer to Section 2.6 below.
\begin{theorem}\label{main2}
Let $M,M^*\subset\CC{2}$ be two real-analytic Levi-nonflat hypersurfaces, both 
of infinite type at $0$, and let $\widehat H:\,(M,0)\mapsto (M^*,0)$ be a formal CR equivalence. 
Then there exist a constant $s>0$ and  local holomorphic coordinates $(z,w)$ for $M,M^*$ at $0$,  at which the complex locus of both 
$M$ and $M^*$ is $\{w=0\}$, a disc $\Delta\subset\CC{}$, and
sectors $S^\pm\subset (\CC{},0)$  with vertex at $0$ containing the directions $\RR{\pm}$, respectively, and holomorphic 
maps   $H_{\pm}\colon \Delta\times S^\pm \to \C^2$ such that $\widehat{H}$
is the $(0,s)$ multi Gevrey asymptotic expansion of $H_{\pm}$ and
$H_{\pm}(M\cap (\Delta\times S^\pm)) \subset M^*$; in particular, $H_{\pm}|_M$
defines a CR diffeomorphism $H$ of $M$ onto $M^*$. 

As a consequence, the given formal power series map $\widehat H(z,w)$ belongs to the $(0,s)$ multi Gevrey class, and thus satisfies \eqref{sbounds} with $r=0$.
\end{theorem}

\begin{remark}\label{multisumty}
In fact,  the proof of \autoref{main2} shows that the formal map $\widehat H$ in \autoref{main2} has the {\em multisummability} property (see Section 2.6 for details).
\end{remark}

\begin{remark}\label{optimality}
As follows from the counter-examples given in \cite{divergence,nonanalytic}, the properties of formal CR-maps stated in \autoref{main2} and \autoref{multisumty} are in general optimal and can't be strengthened further.
\end{remark}

\begin{remark}\label{gevreyord}
It can be verified from the proof of \autoref{main2} that, for $m\geq 2$, the opening of the sectors $S^\pm$ in \autoref{main2} can be chosen to be $\frac{\pi}{m-1}$ for a generic hypersurface $M$ under consideration, and the Gevrey order $s$ can be chosen to be $s=\frac{1}{m-1}$. For $m=1$ one can take $s=0$ (i.e., $\widehat H$ is {\em convergent}), as follows from the result of Juhlin and the second author \cite{jl1}.
\end{remark}


We end this 
introduction by giving 
a short guide to the proof of 
\autoref{main}. As discussed above, the main tool of the paper is the multisummability theory from Dynamical Systems, which meets CR-geometry via the recent
CR \,$\lr$\,DS (Cauchy-Riemann manifolds \,$\lr$\,\,Dynamical Systems) technique developed by Shafikov and the first two 
authors in the recent works \cite{divergence,nonminimalODE,nonanalytic,analytic}. In this 
framework, we study maps of CR-submanifolds $M,M^*$ with 
prescribed properties (such as
being of infinite type) through
symmetries of  an associated holomorphic dynamical system
$\mathcal E(M)$.  
The possibility to replace a real-analytic
CR-manifold by a complex dynamical system is based on the
fundamental parallel between CR-geometry and the geometry of
completely integrable PDE systems. This parallel was first observed by E.~Cartan and
 Segre \cite{cartan,segre} 
and was revisited and further developed in a recent series of publications by
Sukhov (\cite{sukhov1,sukhov2}). The ``mediator''
between a CR-manifold and the associated PDE system is the Segre
family of the CR-manifold. Unlike the Levi-nondegenerate setting in the cited work \cite{cartan,segre,sukhov1,sukhov2}, the CR\,-\,DS technique  deals specifically
 with the {\em Levi-degenerate} setting, providing sort of a dictionary between CR-geometry and Dynamical Systems.

For the proof of
\autoref{main2} we need to develop the CR - DS technique further, extending it to the {\em entire} class of real-analytic hypersurfaces in $\CC{2}$. 
In Section 3, 
infinite type hypersurfaces  satisfying a certain  nondegeneracy assumption (generic infinite type case) are studied.   
 To do so,  we follow the approach in \cite{nonminimalODE,analytic} and consider  complex  meromorphic differential equations associated with these hypersurfaces. Any formal map between real hypersurfaces has to transform the associated ODEs into each other, and working out the latter condition gives a certain {\em singular Cauchy problem} for the components of the map. We then apply the multisummability theory for formal power series solutions of nonlinear systems of ODE at an irregular singularity \cite{Brak1,Ram-Sib2}.   First, we show that  the singular Cauchy  problem has  solutions, holomorphic in certain sectorial domains with  Gevrey asymptotic expansion. Second, we show that these solutions have certain uniqueness properties giving the condition $H(M)\subset M^*$ for the arising CR-map defined on $M$.

In Section 4, we have to extend the scheme in Section 3 to the exceptional (non-generic) case. For doing so, we introduce a new tool: {\em associated differential equations of high order}. In turns out that {\em any} Levi-nonflat real-analytic hypersurface $M$ (in particular, a finite type hypersurface!) can be associated, in appropriate local holomorphic coordinates, with a system of singular ODEs of the kind \eqref{assocd}. We  achieve this by a sequence of coordinate changes and appropriate {\em blow-ups} (both in the initial space and in the space of parameters for Segre families). In this regard, the blow-up procedure of Mir and the second author from \cite{lmblowups} is a key tool. The initial formal CR map is shown to be a transformation between the associated systems of singular ODE again. Working out the transformation rule here brings significant new  difficulties, since  we  deal with jet prolongations of arbitrarily high order. After overcoming these difficulties, we are again able to apply the multisummability theory and obtain the desired regularity property for the formal CR map.

\begin{center} \bf Acknowledgements  \end{center}

\medskip

The authors would like to thank Nordine Mir for his valuable remark on the possibility to obtain the assertion of \autoref{mir} from \autoref{main}. 

\section{Preliminaries}

\subsection{Segre varieties.}
Let $M$ be a smooth  real-analytic submanifold in $\cx^{n+k}$ of
CR-dimension $n$ and CR-codimension $k$, $n,k>0$, $0\in M$, and
$U$ a neighbourhood of the origin where $M\cap U$ admits a
real-analytic defining function $\phi(Z,\overline Z)$ with
the property that $\phi (Z,\zeta)$ is a holomorphic function for
for $(Z, \zeta) \in U\times \bar U$. For every
point $\zeta\in U$ we associate its Segre
variety in $U$ by
$$
Q_\zeta= \{Z\in U : \phi(Z,\overline \zeta)=0\}.
$$
Segre varieties depend holomorphically on the variable $\overline
\zeta$, and for small enough neighbourhoods $U$ of $0$, they
are actually holomorphic submanifolds of $U$ of codimension $k$.

One can choose coordinates $Z = (z,w) \in \CC{n} \times \CC{k} $ and
 a neighbourhood
$U={\
U^z}\times U^w\subset \cx^{n}\times \CC{k}$ such that, for any $\zeta\in U,$
$$
  Q_\zeta=\left \{(z,w)\in U^z \times U^w: w = h(z,\overline \zeta)\right\}
$$
is a closed complex analytic graph. $h$ is a holomorphic
function on $U^z \times \bar U$.  The antiholomorphic $(n+k)$-parameter family of complex
submanifolds $\{Q_\zeta\}_{\zeta\in U_1}$ is called  {\em the Segre
family} of $M$ at the origin. The following basic
properties of Segre varieties follow from the definition and the
reality condition on the defining function:
\begin{equation}\label{e.svp} \begin{aligned}
  Z\in Q_\zeta & \Leftrightarrow  \zeta\in Q_Z,
 \\
  Z\in Q_Z & \Leftrightarrow  Z\in M,
\\
 \zeta\in M & \Leftrightarrow \{Z\in U \colon Q_\zeta=Q_Z\}\subset M.
\end{aligned}
\end{equation}

The fundamental role of Segre varieties for holomorphic maps
is due to their {\em invariance property}: If $f: U \to U'$ is a
holomorphic map which sends a smooth real-analytic submanifold
$M\subset U$ into another such submanifold $M'\subset U'$, and $U$
is chosen as above (with the analogous choices and
notations for $M'$), then
$$
 f(Q_Z)\subset Q'_{f(Z)}.
$$
For
more details and other properties of Segre varieties we refer the reader
to
e.g. \cite{webster}, \cite{DiPi},  or \cite{ber}.


The space of Segre varieties $\{Q_Z : Z\in U\}$,
for appropriately chosen $U$, can be
identified with a subset of $\cx^K$ for some $K>0$ in such a way
that the so-called {\em Segre map} $\lambda : Z \to Q_Z$ is
antiholomorphic. This  can be seen from the fact that if we write
\[ h(z,\bar \zeta) = \sum_{\alpha\in\mathbb{N}^n} h_\alpha (\bar \zeta) z^\alpha, \]
then $\lambda (Z)$ can be identified with
$\left(h_\alpha (\bar Z) \right)_{\alpha\in\mathbb{N}^n}$. After that the desired fact follows from the Noetherian property.

If $M$ is a hypersurface,
then its Segre map is one-to-one in a
neighbourhood of every point $p$ where $M$ is Levi nondegenerate.
When such a real hypersurface $M$ contains a complex hypersurface
$X$, for any point $p\in X$ we have $Q_p = X$ and $Q_p\cap
X\neq\emptyset\Leftrightarrow p\in X$, so that the Segre map
$\lambda$ sends the entire $X$ to a unique point in $\CC{N}$ and,
accordingly, $\lambda$ is not even finite-to-one near each $p\in
X$, i.e. $M$ is \it not essentially finite \rm at points $p\in
X$. For the notion of essential finiteness, see e.g. \cite{ber}.

\subsection{Nonminimal real hypersurfaces} 
We recall that given a real-analytic Levi-nonflat hypersurface $M\subset \CC{2}$, for 
every $p\in M$ there exist so-called {\em normal coordinates} $(z,w)$ centered
at $p$, i.e. a local holomorphic coordinate system near $p$ in 
which $p=0$ and near $0$, $M$ is defined by an equation of the form 
\[ v = F (z, \bar z , u)\] 
for some germ $F$ of a holomorphic function on $\CC{3}$ which satisfies 
\[ F(z,0,u) = F(0,\bar z, u) = 0\]
and the 
{\em reality condition} $F(z,\bar z,u) \in \RR{} $ for $(z,u)\in \CC{} \times\RR{}$ close to $0$ (see e.g. \cite{ber}). 

We say that $M$ is {\em nonminimal} at $p$ if there exists a germ of a nontrivial
complex curve $X\subset M$ through $p$. It turns out that in normal coordinates, such 
a curve $X$ is necessarily defined by $w = 0$; 
in particular, any such $X$ is nonsingular.

Thus a Levi-nonflat hypersurface $M$ is nonminimal if and only if with normal coordinates $(z,w)$ and a 
defining function $F$ as above, we have that $F(z,\bar z,0) = 0$, or equivalently, 
if $M$ can defined by an equation of the form 
\begin{equation}\label{mnonminimal}
v=u^m\psi(z,\bar z,u), 
\text{ with } \psi(z,0,u) = \psi(0,\bar z, u)= 0 \text{ and } \psi(z,\bar z,0)\not\equiv 0,
\end{equation}
where $m\geq 1$.
 
It turns out that the integer $m\geq 1$ is independent of the choice
of normal coordinates  (see \cite{meylan}), and actually also 
of the choice of $p\in X$; we refer to $m$ as the 
{\em nonminimality order} of a Levi-nonflat hypersurface $M$ on $X$ (or at $p$) and say that
$M$ is {\em $m$-nonminimal} along $X$ (or at $p$).

Several other variants of defining functions for $M$ are useful. Throughout this paper, we use
the {\em complex defining function} $\Theta$ in which $M$ is defined by 
\[  w = \Theta (z,\bar z,\bar w) ;\]
it is obtained from $F$ by solving the equation 
\[ \frac{w - \bar w}{2i} = F \left(z,\bar z, \frac{w+\bar w}{2} \right) \]
for $w$. The complex defining function satisfies the conditions 
\[ \Theta(z,0,\eta) = \Theta (0, \xi, \eta ) = \tau, \quad \Theta(z,\xi,\bar \Theta (\xi, z, w)) = w. \]
If $M$ is $m$-nonminimal at $p$, then $\Theta (z,\xi,\eta) = \eta \theta (z,\xi,\eta)$ and 
thus $M$ is defined by 
\[ w = \bar w \theta(z,\bar z,\bar w) = \bar w (1 +\bar w^{m-1}\tilde \theta(z,\bar z,\bar w) ), \text{ where } 
\tilde \theta (z,0,\eta) = \tilde \theta (0,\xi, \eta) = 0 \text{ and }  
\tilde\theta (z,\xi,0)\neq 0.\]

The {Segre family} of $M$, where $M$ is given in 
normal coordinates as above, with the complex defining function 
$\Theta\colon U_z\times \bar U_z \times \bar U_w = U_z \times \bar U \to U_w$ consists of the complex hypersurfaces $Q_\zeta \subset U$, defined for  $\zeta \in U$
by 
\[ Q_\zeta = \{(z,w) \colon w = \Theta (z,\bar \zeta)\}. \]
The real line 
\begin{equation}\label{Gamma}
\Gamma = \{(z,w)\in M \colon z = 0 \} = \{(0,u) \in M \colon u\in\RR{}\}\subset M
\end{equation}
has the property that 
\[ Q_{(0,u)} = \{w = u\}, \quad (0,u)\in \Gamma \]
for $u\in\RR{}$, a property which actually is equivalent to the normality of 
the coordinates $(z,w)$. More exactly, for any real-analytic curve 
$\gamma$ through $p$ one can find normal coordinates $(z,w)$ in 
which $\gamma$ corresponds to $\Gamma$ in \eqref{Gamma} (see e.g. \cite{lmblowups}). 

We finally have to point out that  a real-analytic Levi-nonflat hypersurface $M\subset\CC{2}$ can exhibit nonminimal points of two kinds, which can be  referred to as  {\em generic} and {\em exceptional} nonminimal points, respectively. A generic point $p\in M$ is characterized by the condition that {\em the minimality locus $M\setminus X$ of $M$ is Levi-nondegenerate locally near $p$}. At a generic nonminimal point, \eqref{mnonminimal} is supplemented by the condition $\psi_{z\bar z}(0,0,0,)\neq 0.$ In terms of the complex defining function, it gives the following useful representation for $M$:
\begin{equation}\label{admissible}
w=\Theta(z,\bar z,\bar w)=\bar w+\bar w^m\sum_{k,l\geq 1}\Theta_{kl}(\bar w)z^k\bar z^l, \quad \Theta_{11}(0)\neq 0
\end{equation}  
(see, e.g., \cite{nonminimalODE}). 

If, otherwise, the intersection of   the minimal locus $M\setminus X$ of $M$ with any neighborhood of $p$ in $M$ contains Levi-degenerate points, then such a point $p$ is referred to as {\em exceptional}.

\subsection{Real hypersurfaces and second order differential equations.}\label{sub:realhyp2ndorderequ}
To every Levi nondegenerate real hypersurface
$M\subset\CC{N}$  we can associate a system of second order
holomorphic PDEs with $1$ dependent and $N-1$ independent
variables, using the Segre family of the hypersurface. This remarkable construction
 goes back to
E.~Cartan \cite{cartan} and Segre \cite{segre} (see also a remark by Webster \cite{webster}),
and was recently revisited in the work of Sukhov
\cite{sukhov1},\cite{sukhov2} in the nondegenerate setting, and in the work of Shafikov and the first two authors in the degenerate setting (see\cite{divergence},\cite{nonminimalODE},\cite{nonanalytic},\cite{analytic}). 
We  describe this procedure in the case
$N=2$ relevant for our purposes.

Let $M\subset\CC{2}$ be a smooth real-analytic
hypersurface, passing through the origin, and $U = U_z \times U_w$
 a sufficiently small neighborhood of the origin. In this case
we associate a second order holomorphic ODE to $M$, which is uniquely determined by the condition that the equation is satisfied by all the
graphing functions $h(z,\zeta) = w(z)$ of the
Segre family $\{Q_\zeta\}_{\zeta\in U}$ of $M$ in a
neighbourhood of the origin.

More precisely, since $M$ is Levi-nondegenerate
near the origin, the Segre map
$\zeta\lr Q_\zeta$ is injective and the Segre family has
the so-called transversality property: if two distinct Segre
varieties intersect at a point $q\in U$, then their intersection
at $q$ is transverse. Thus, $\{Q_\zeta\}_{\zeta\in U}$ is a
2-parameter  family of holomorphic
curves in $U$ with the transversality property, depending
holomorphically on $\bar\zeta$. It follows from
the holomorphic version of the fundamental ODE theorem (see, e.g.,
\cite{ilyashenko}) that there exists a unique second order
holomorphic ODE $w''=\Phi(z,w,w')$, satisfied by all the graphing functions of
$\{Q_\zeta\}_{\zeta\in U}$.

To be more explicit we consider the complex defining equation  $w=\rho(z,\bar z,\bar w)$, as introduced above. 
 The Segre
variety $Q_\zeta$ of a point $\zeta=(a,b)\in U$ is  now given
as the graph
\begin{equation} \label{segre0}w (z)=\rho(z,\bar a,\bar b). \end{equation}
Differentiating \eqref{segre0} once, we obtain
\begin{equation}\label{segreder} w'=\rho_z(z,\bar a,\bar b). \end{equation}
Considering \eqref{segre0} and \eqref{segreder}  as a holomorphic
system of equations with the unknowns $\bar a,\bar b$, an
application of the implicit function theorem yields holomorphic functions
 $A, B$ such that
$$
\bar a=A(z,w,w'),\,\bar b=B(z,w,w').
$$
The implicit function theorem applies here because the
Jacobian of the system coincides with the Levi determinant of $M$
for $(z,w)\in M$ (\cite{ber}). Differentiating \eqref{segreder} once more
and substituting for $\bar a,\bar b$ finally
yields
\begin{equation}\label{segreder2}
w''=\rho_{zz}(z,A(z,w,w'),B(z,w,w'))=:\Phi(z,w,w').
\end{equation}
Now \eqref{segreder2} is the desired holomorphic second order ODE
$\mathcal E = \mathcal{E}(M) $.

More generally, the association of   a completely integrable PDE  with
a CR-manifold is possible for a wide range of
CR-submanifolds (see \cite{sukhov1,sukhov2}). The
correspondence $M\lr \mathcal E(M)$ has the following fundamental
properties:

\begin{enumerate}

\item[(1)] Every local holomorphic equivalence $F:\, (M,0)\lr (M',0)$
between CR-submanifolds is an equivalence between the
corresponding PDE systems $\mathcal E(M),\mathcal E(M')$ (see \autoref{sub:equiv2ndorder});

\medskip

\item[(2)] The complexification of the infinitesimal automorphism algebra
$\mathfrak{hol}^\omega(M,0)$ of $M$ at the origin coincides with the Lie
symmetry algebra  of the associated PDE system $\mathcal E(M)$
(see, e.g., \cite{olver} for the details of the concept).

\end{enumerate}

Even though for a real hypersurface
$M\subset\CC{2}$ which is nonminimal at the origin there is no a priori way to associate
to $M$ a second order ODE or even a more general PDE system near
the origin,  the Shafikov and the first author
found an injective correspondence between  nonminimal hypersurfaces $M\subset\CC{2}$ which are spherical outside the complex locus hypersurfaces
 and certain {\em singular} complex ODEs $\mathcal E(M)$ with an
isolated  singularity at the origin in in \cite{nonminimalODE}. It is possible to extend this construction to the non-spherical case, which we do in Section 3.



\subsection{Equivalences and symmetries of  ODEs}\label{sub:equiv2ndorder}
We start with a description of the jet prolongation approach to
the equivalence problem (which is a simple interpretation of a
more general approach in the context of {\em jet bundles}). We refer to the excellent sources \cite{olver}, \cite{bluman} for more details and collect the necessary 
prerequisites here.
In what follows all variables are assumed to be complex, all
mappings biholomorphic, and all ODEs to be defined near their zero
solution $y(x)=0$.

Consider two ODEs,  $\mathcal E$ given by $y^{(k)}=\Phi(x,y,y',...,y^{(k-1)} )$
and
$\tilde{\mathcal E}$ given by $ y^{(k)}=\tilde\Phi(x,y,y',...,y^{(k-1)})$, where the functions
$\Phi$ and $\tilde\Phi$ are holomorphic in some neighbourhood of the
origin in $\CC{k+1}$. We say that a germ  of a biholomorphism
$H \colon (\CC{2},0)\lr(\CC{2},0)$  transforms $\mathcal E$ into
$\tilde{\mathcal E}$, if it sends (locally) graphs of solutions of
$\mathcal E$ into graphs of solutions of $\tilde{\mathcal E}$.
We define the {\em $k$-jet space} $J^k(\CC{},\CC{})$ to be the $(k+2)$-dimensional
 linear space with coordinates $x,y,y_1,...,y_{k}$,
which correspond to the independent variable $x$, the dependent
variable $y$ and its derivatives up to order $k$, so that we can
naturally consider $\mathcal E$ and $\tilde{\mathcal E}$ as complex
submanifolds of $J^k(\CC{},\CC{})$.

For any biholomorphism $H$ as above one may consider its
{\em $k$-jet prolongation} $H^{(k)}$, which is defined on a neighbourhood of the origin in $\CC{k+2}$ as follows.
The first two components of the mapping $H^{(k)}$
coincide with those of $H$. To obtain the remaining components  we
denote the coordinates in the preimage by $(x,y)$ and in the
target domain by $(X,Y)$. Then the derivative $\frac{dY}{dX}$ can
be symbolically recalculated, using the chain rule, in terms of
$x,y,y'$, so that the third coordinate $Y_1$ in the target jet
space becomes a function of $x,y,y_1$. In the same manner one
obtains the remaining components of the prolongation of the
mapping $H$. Thus, for differential equations of order $k$, {\em a mapping $H$ transforms the ODE $\mathcal E$ into $\tilde{\mathcal
E}$ if and only if the prolonged mapping $H^{(k)}$ transforms
$(\mathcal E,0)$ into $(\tilde{\mathcal E},0)$ as submanifolds in the
jet space $J^k(\CC{},\CC{})$}. A similar statement can be formulated for systems of differential equations, as well as for
certain singular differential equations, for example, the ones considered in the next subsection. 

Some further details and properties of the jet prolongations  $H^{(k)}$ are given in Section 4.




 \smallskip
 

\subsection{Tangential sectorial domains and smooth CR-mappings}
Let $M\subset\CC{2}$ be a real-analytic  Levi nonflat hypersurface, which is  nonminimal at a point $p\in M$, and $X\ni p$ its complex locus.  We choose for $M$ local holomorphic coordinates \eqref{mnonminimal} so that $p=0,\,X=\{w=0\}$.
We next recall the following definition (see \cite[Section 1.2]{analytic}).
\begin{definition}\label{tangential} A set $D_p\subset\CC{2},\,D_p\ni p$ is called a \em tangential sectorial domain for $M$ at $p$  \rm if, in some   local holomorphic coordinates   $(z,w)$ for $M$ as above, the set $D_p$ looks as 
\begin{equation}\label{standard}
\Delta\times \Bigl(S^+\cup\{0\}\cup S^-\Bigr).
\end{equation}
Here $\Delta\subset\CC{}$ is a disc of radius $r>0$, centered at the origin, and $S^\pm\subset\CC{}$ are sectors
\begin{equation}\label{sectors}
S^+=\bigl\{|w|<R,\, \alpha^+<\mbox{arg}\,w<\beta^+\bigr\},\quad S^-=\bigl\{|w|<R,\, \alpha^-<\mbox{arg}\,w<\beta^-\bigr\}
\end{equation}
\noindent for appropriate  $R>0$ and such that $S^{\pm}$ contains the direction $\R^{\pm}$.
We also denote by $D^\pm_p$ the domains $\Delta\times S^\pm\subset\CC{2}$ respectively.
\end{definition} 
As discussed in \cite{analytic}, 
{\em for any tangential sectorial domain $D_p$ for $M$ at $p$,  the intersection of $M$ with a sufficiently small neighborhood $U_p$ of $p$ in $\CC{2}$ is contained in $D_p$}.

\smallskip

Next, we recall the following classical notion.
\begin{definition}\label{poincare}
Let $f(w)$ be a function holomorphic in a sector $S\subset\CC{}$. We say that a formal power series $\hat f(w)=\sum_{j\geq 0}c_jz^j$ is the {\em Poincar\'e asymptotic expansion of $f$ in $S$}, if for any $n\geq 0$ we have:
$$\frac{1}{w^n}\left(f(w)-\sum_{j=0}^n c_jw^j\right)\rightarrow 0\quad\mbox{when}\quad w\rightarrow 0,\,\,w\in S.$$ 
In the latter case, we write: $f(w)\sim \hat f(w)$.
\end{definition}
For basic properties of the asymptotic expansion we refer to \cite{vazow}. In particular, we recall that asymptotic expansion in a full punctured neighborhood of a point means the usual holomorphicity of a function. 

The notion of Poincar\'e asymptotic expansion can be naturally extended to function holomorphic in products of sectors and the respective formal power series in several variables. This allows us to formulate the following 
\begin{definition}\label{asymptotic}
 We say that 
 a $C^\infty$ CR-function $f$ in a neighborhood of $p$ in $M$ is {\em sectorially extendable,} if for some (and then any sufficiently small) tangential sectorial domain $D_p$ for $M$ at $p$, there exist functions $f^\pm\in\mathcal O(D^\pm_p)$ such that 
 
 \smallskip
 
 $(i)$ each $f^\pm$ coincides with $f$ on $D^\pm_p\cap M$, and 
 
 \smallskip
 
$ (ii)$ both $f^\pm$ admit the same  Poincar\'e asymptotic representation
 $$ f^\pm\sim \sum_{k,l\geq 0}a_{kl}z^kw^l \in \fps{z,w}$$
 in the respective domains $D^\pm_p$. 
  \end{definition}
 
 \smallskip

 We can similarly define the sectorial extendability of CR-mappings or infinitesimal CR-automorphisms of real-analytic hypersurfaces. Crucially, it is not difficult to see (as discussed in \cite{analytic}) that {\em restricting two holomorphic functions $f^\pm$, as in \autoref{asymptotic}, onto a nonminimal hypersurface $M$ as above defines a $C^\infty$ CR-function on $M$ near $0$, sectorially extendable into the initial tangential sectorial domain.} This observation will be the final ingredient for the proof of \autoref{main}.
    
\subsection{Summability of formal power series}
In this section, we shall recall some known facts about multisummability of formal power series and we shall recall a key theorem due to Braaksma that says that any formal solution of a system of nonlinear differential equations at an irregular singularity is multisummable in any direction but a finite number of them. This means there are holomorphic solutions in some sectors with vertex at the singularity and having the formal solution as asymptotic power series. This has a long although recent history and we refer to \cite{ramis-ksum, ramis-panorama, balser-book, sibuya-textbook, ramis-stolo-cours} for more information.

\begin{definition}
Let $s>0$. A formal power series $\hat f=\sum_{n\geq 0}f_n z^n$ is said to be a {\em Gevrey series of order $s$} if there exist $A,B>0$ suh that $|f_n|\leq AB^n\Gamma(1+sn)$ for all $n$. The space of such power series is denoted by $\fps{z}_s$.
\end{definition}
In other words, we have $|f_n|\leq \tilde A\tilde B^n(n!)^s$ for some appropriate constants.
Let $I=]a,b[$ be an open interval of $\R$ and let $r>0$. We denote by ${\mathcal S}_r(I)$ the open sector of $\C$(or the Riemann surface of the Logarithm)~:
$$
{\mathcal S}_r(I):=\{z\in \C|\quad a<\arg z <b, \quad 0<|z|<r\}.
$$
\begin{definition}\label{gevreyexp}
A holomorphic function $f\in {\mathcal O}({\mathcal S}_r(I))$ is said to have an {\em $s$-Gevrey asymptotic expansion} at $0$ if there exists a formal power series $\hat f=\sum_{j\geq 0}f_j z^j$ such that, for all $I'\subset\subset I$, there exist $C>0$ and $0<r'\leq r$ such that for all integer $n>0$
$$
\left|f(z) -\sum_{k=0}^{n-1}f_j z^j \right|\leq C^n\Gamma(1+sn)|z|^n,\quad \forall z\in {\mathcal S}_{r'}(I').
$$
We shall write $f\sim_s \hat f$. The space of these functions will be denoted by ${\mathcal A}_s(I)$.
\end{definition}
Note that the above Gevrey asymptotic property strengthens the Poincar\'e asymptotic property introduced in the previous section.  We also remark that asymptotic series $\hat f$ of such a function belongs to $\fps{z}_s$.
\begin{definition}
Let $k$ be a positive real number. A formal power series $\hat f\in \fps{z}_{\frac{1}{k}}$ is said to be {\em $k$-summable} in the direction ${d}$ if there exists a sector ${\mathcal S}_r(I)$, bissected by ${d}$ and of opening $|I|>\frac{\pi}{k}$, and a holomorphic function $f\in {\mathcal O}({\mathcal S}_r(I))$ such that $f\sim_{\frac{1}{k}} \hat f$. We also say that $\hat f$ is $k$-summable on $I$.
\end{definition}
Such a holomorphic function $f$ is unique (this is a consequence of Watson Lemma \cite{malgrange-cours}) and called the { \em $k$-sum of $\hat f$}. We emphasize that a {\em $k$-summable power series is $\frac{1}{k}$-Gevrey}. In order to describe the properties of solutions of differential equations with irregular singularity, we need the more general notion of multi-summability. 
\begin{definition}
Let $r\geq 1$ be an integer and let ${\bf k}:=(k_1,\ldots,k_r)\in (\R)^r$ with $0<k_1<\cdots<k_r$. For any $1\leq j\leq r$, let $I_j:=]a_j,b_j[$ be an open interval of length $|I_j|=b_j-a_j>\frac{\pi}{k_j}$ such that $I_j\subset I_{j-1}$, $2\leq j\leq r$. A formal power series $\hat f\in \fps{z}$ is said to be {\em ${\bf k}$-multisummable on ${\bf I}=(I_1,\ldots,I_r)$} if there exist formal power series $\hat f_j$ such that $\hat f:= \sum_{j=1}^r \hat f_j$ and such that each $\hat f_j$ is $k_j$-summable on $I_j$ with sum $f_j$, $1\leq j\leq r$. We shall also say that {\em $\hat f$ is ${\bf k}$-multisummable in the multidirection ${\bf d}=(d_1,\ldots, d_r)$} where $d_j$ bissects the sector $\{a_j<\arg z<b_j\}$. 
\end{definition}
In that case, we say that {\em ${\bf f}=(f_1,\ldots,f_r)$ is the multisum of $\hat f$}. Such a multisum is unique according the relative Watson lemma \cite{malgrange-cours}[Th\'eor\`eme 2.2.1.1]. From it, one can build the (unique) ${\bf k}$-sum of $\hat f$ on {\bf I}, denoted by ${\bf f_{k,I}}$, that satisfies ${\bf f_{k,I}}\sim_{\frac{1}{k_1}} \hat f$ on $I_1$\cite{Brak1}[p.524]. Here we have  used the definition of W. Balser \cite{balser-multi} but there are other equivalent definitions due to Ecalle\cite{EcalleIII,Ram-Mart-acc} and Malgrange-Ramis\cite{Malg-Ram}.

Next, we shall emphasize the folowing important property:
\begin{proposition}\cite[Proposition 3.2,p. 358]{Malg-Ram}, \cite[Th\'eor\`eme 2.2.3.1]{malgrange-cours}\label{compo-multi}
Let $\Phi$ be a germ of holomorphic function at $0$ of $\C^{p+1}$. Let  $\hat f_i\in \fps{z}$ be a formal power series such that $\hat f_i(0)=0$, $i=1,\ldots, p$.
Assume that $\hat f_i$ is ${\bf k}$-multisummable on ${\bf I}=(I_1,\ldots,I_r)$ with multisum  ${\bf f_i}=(f_{i,1},\ldots,f_{i,r})$. Then, $\Phi(z,\hat f_1(z),\ldots, \hat f_p(z))$ is also ${\bf k}$-multisummable on ${\bf I}=(I_1,\ldots,I_r)$ with multisum $\Phi(z,{\bf f})=(\Phi(z,f_{1,1},\ldots,f_{p,1} ),\ldots,\Phi(z,f_{1,r},\ldots,f_{p,r}))$.
\end{proposition}
In particular, we conclude that the class of multisummable functions forms an algebra and is closed under the division operation, provided the denominator has no constant terms in its expansion.

The reason for introducing these notions is that these are the natural spaces to which solutions of nonlinear differential equations with irregular singularity must belong. 

Let $r\in \N$, $k_j\in \N$, $j=1,\ldots ,r$, $0<k_1<\ldots <k_r$. We set ${\bf k}:=(k_1,\ldots,k_r)$. Let ${\bf I}=(I_1,\dots ,I_r)$ where $I_j=]\alpha_j,\beta_j[$ is an open interval with $\beta_j -\alpha_j>\pi/k_j$. We also assume that $I_j\subset I_{j-1}, j=1,\ldots ,r$ where $I_0=\R$.
Consider
\begin{equation}\label{mr}
\text{diag}\{x^{k_1} I^{(1)},\ldots ,x^{k_r} I^{(r)}\}x\frac{dy}{dx}=\Lambda y +xg(x,y)
\end{equation}
where $I^{(j)}$ denotes the identity matrix of dimension $n_j\in \N$ and $n=n_1+\ldots n_r$, $y\in \C^n$, $\Lambda=\text{diag}\{\lambda_1,\ldots, \lambda_n\}$, $\Lambda$ is invertible and $g$ is a {\bf k}-sum of some $\hat g(x,y)\in \C[[x,y]]$ on {\bf I} uniformly in a neighborhood of $0\in \C^n$ ($g$ analytic at $(0,0)$ in $\C \times \C^n$ is a special case). Let $\hat{y}=\sum_{h=1}^\infty c_hx^h$ be a formal solution of (\ref{mr}). This means that 
$$
\text{diag}\{x^{k_1} I^{(1)},\ldots ,x^{k_r} I^{(r)}\}x\frac{d \hat y(x)}{dx}=\Lambda y +x\hat g(x,\hat y(x)).
$$
Then the following holds (cf. 
\cite{Ram-Sib2,Brak1,BBRS})
 \begin{theorem}\label{final-multisum}\cite{braaksma-banach}
The formal solution $\hat{y}$ of (\ref{mr}) is {\bf k}-multisummable on ${\bf I}=(I_1,\dots ,I_r)$
 if $\arg\lambda_h\not\in ]\alpha_j +\pi/(2k_j),\beta_j -\pi/(2k_j)[$ for all 
$h\in [n_1+\ldots +n_{j-1}+1,n_1+\ldots +n_j]$, $1\leq j\leq r$.
\end{theorem}
\begin{corollary}\cite{Brak1}[Corollary p.525]\label{general-multisum}
Consider an analytic nonlinear differential equation of the form
\begin{equation}
z^{\nu+1}\frac{d y}{dz}=F(z,y)
\label{irreg-sing}
\end{equation}
where $z\in \C$, $y\in \C^n$, and $F$ is analytic in a neighborhood of the the origin in $\C\times \C^n$, $\nu>0$. Then, there exist   positive integers $q$ and $0<k_1<\ldots <k_r$ such that  every formal power series solution $\hat y$ of (\ref{irreg-sing}) is $(\frac{k_1}{q},\ldots, \frac{k_r}{q})$-multisummable.
\end{corollary}
As shown in \cite{Brak2}[p.60], there exists an analytic transformation and a ramification $x=z^{1/q}$ which transforms (\ref{irreg-sing}) into (\ref{mr}). As a consequence, we can also apply \autoref{final-multisum} in the situation when  the righthand side $F$  is a $(\frac{k_1}{q},\ldots, \frac{k_r}{q})$-sum, uniformly in a neighborhood of $0\in \C^n$. The point for not stating this directly in the theorem is that both ${\bf k}$ and $q$ need to be known and cannot be read off immediately on (\ref{irreg-sing}).
\begin{remark}\label{real-axis}
Let $\hat y$ be a formal power series solution of (\ref{irreg-sing}). Let ${\bf k/q}:= (\frac{k_1}{q},\ldots, \frac{k_r}{q})$ as above. Then $\hat y$ has a ${\bf k/q}$-sum $y^{\pm}$ defined in a sector containing the direction $\R^{\pm}$. Indeed, 
having done an appropriate analytic transformation and a ramification $x=z^{1/q}$, we consider (\ref{mr}). Let $\epsilon_j>0$ and let $\tilde I:=\cup_j\cup_{h\in [n_1+\ldots +n_{j-1}+1,n_1+\ldots +n_j]}]\arg \lambda_h-\epsilon_j, \arg \lambda_h+\epsilon_j[$. It is always possible to choose the $\epsilon_j$'s small enough so that the exists a $\tau_+\not\in\tilde I$ and so that $|\tau_+|<\frac{\pi}{2k_r} +\frac{1}{2}\min\frac{\epsilon_j}{2}$. Therefore, for all $j$, $-\tau_+-\frac{\pi}{2k_j} -\frac{\epsilon_j}{2}<0<-\tau_++\frac{\pi}{2k_j} +\frac{\epsilon_j}{2}$. This means that $\R^+$ belongs to the sector $I_j^+$ bissected by $\tau_+$ and of opening $\frac{\pi}{k_j} +\epsilon_j$, for all $j$. Setting $\tau_-=\tau_++\pi$, then $\R^-$ belongs to the sector $I_j^-$ bissected by $\tau_-$ and of opening $\frac{\pi}{k_j} +\epsilon_j$, for all $j$. According to theorem \autoref{final-multisum}, $\hat y$ is ${\bf k}$-multisummable on ${\bf I}^{\pm}$ and its ${\bf k}$-sum ${\bf y_{k,I^{\pm}}}$ is defined on $\R^{\pm}$. To obtain the same result for \ref{irreg-sing}, one has to divide $\tau_+$ by $q$ and set $\tau_-=\tau_++\pi/q$.
\end{remark}

\section{Complete system for a generic nonminimal hypersurface}

We start with the proof of \autoref{main}. We assume both reference points $p,p^*$ to be the origin. As was discussed in the Introduction, in the finite type case the assertion of \autoref{main} follows from \cite{ber1}. In the Levi-flat case the assertion is obvious. Hence, we assume in what follows that {\em both $M,M^*$ are nonminimal at the reference point $0$ but are Levi-nonflat}. 

In this section, we prove \autoref{main} for the class of $m$-nonminimal at the origin hypersurfaces, satisfying the generic assumption that {\em the minimal part $M\setminus X $ of $M$ is Levi-nondegenerate} (thus the origin is a generic nonminimal points, in the terminology of Section 2).
As was explained in Section 2, any such hypersurface can be written in appropriate local holomorphic coordinates by  an equation \eqref{admissible}.  For the convenience of the reader, 
we shall first show how the individual results proved in this section combine to yield 
the proof of \autoref{main}. 

\subsection{Outline of the results and conclusion of the proof in the generic case}  
In order to state the (technical) results needed, we first need a convenient prenormalization of formal 
power series maps. 
Let us observe that given  two hypersurfaces $M,M^*\subset\CC{2}$, given near the origin by  \eqref{admissible} then any  formal power series map 
$$H=(F,G):\quad(M,0)\mapsto (M^*,0)$$ between them has the following specific form. 
\begin{lemma}\label{specialmap}
Any formal power series map $$(z,w)\mapsto \bigl(F(z,w),G(z,w)\bigr)$$ between germs at the origin of two hypersurfaces of the form \eqref{admissible} satisfies:
\begin{equation}\label{specialg}
G= O(w), \quad  G_z=O(w^{m+1}).
\end{equation}
\end{lemma}

\begin{proof}
We interpret \eqref{admissible} as:
$$w=\bar w+\bar w^{m}\cdot z\bar z \cdot O(1).$$ 
Then the basic identity gives:
$$G(z,w)=\bar G (\bar z,\bar w)+\bar G^m(\bar z,\bar w)\cdot \bar F(\bar z,\bar w) \cdot F( z, w)\cdot O(1),\,\,\, \mbox{where}\,\,\, w=\bar w+\bar w^{m}\cdot z\bar z\cdot O(1).$$ 
Putting in the latter identity $\bar z=\bar w=0$, we get $G(z,0)\equiv 0$.  Further,  differentiating with respect to $z$, evaluating at $\bar z=0$ at which one has $w=\bar w$, we get:
$$G_z(z, \bar w)= \bar G(0,\bar w)^m\cdot  \bar F(0,\bar w) \cdot O(1),$$
which already implies the assertion of the lemma.
\end{proof}   

\autoref{specialmap} immediately implies that, when considering formal invertible mappings between hypersurfaces of the form \eqref{admissible}, 
{\em we can restrict to transformations of the form:}
$$z\mapsto z+f(z,w), \quad w\mapsto w+wg_0(w)+w^mg(z,w)$$
with
\begin{equation}\label{normalmap}
\begin{aligned}
 f_z(0,0)=0,\quad g_0(0)=0, \quad  g(z,w)=O(zw)
  \end{aligned}
\end{equation} 
(normalizing the coefficients of  $z,w$ for $F,G$ respectively is possible by means of a linear scaling applied to the source hypersurface). We now expand $f,g$ as:
\begin{equation}\label{expandfg}
f(z,w)=\sum_{j=0}^\infty f_j(w)z^j,\quad g(z,w)=\sum_{j=1}^\infty g_j(w)z^j
\end{equation}
(we point out that the function $g_0(w)$, as in \eqref{normalmap}, is {\em not} present in the expansion \eqref{expandfg}!). In view of \eqref{normalmap} we have
\begin{equation}\label{initdata}
f_1(0)=g_1(0)=0.
\end{equation}
We also introduce the new functions
\begin{equation}\label{YY}
y_1:=f_0,\quad y_2:=g_0,\quad y_3:=f_1,\quad y_4:=g_1,\quad  y_5:=w^mf'_0,\quad y_6:=wg'_0,\quad y_7:=w^mf'_1,\quad y_8:=w^mg'_1.
\end{equation}
It is {\em important} that all the $y_j$ do not have a constant term, as follows from \eqref{normalmap},\eqref{initdata}  and the fact that our transformation maps the origin to itself.
We clearly have
\begin{equation}\label{first}
w^my_1'=y_5, \quad wy_2'=y_6, \quad w^my_3'=y_7, \quad w^my_4'=y_8. 
\end{equation}
We can now state the first main technical result of this section: 
\begin{proposition}\label{systemY}
The formal vector function $Y_0(w):=(y_1(w),...,y_8(w))$ satisfies a meromorphic differential equation  
\begin{equation}\label{merom}
w^m\frac{dY}{dw}=A(w,Y),
\end{equation}
where $A(w,Y)$ is a holomorphic at the origin function.
\end{proposition}
Applying now the fundamental  \autoref{final-multisum} on the multisummability of formal solutions of nonlinear differential equation at an irregular singularity, as well as \autoref{real-axis} (see Section 2.6), we immediately obtain
\begin{corollary}\label{summable}
There exist sectors $S^+,S^-\subset\C$, containing the positive and the negative real lines, directions $d^\pm$, functions $f^\pm_0(w),g^\pm_0(w),f^\pm_1(w),g^\pm_1(w)$ holomorphic in the respective sectors, and  a multi-order  $\mathbf{k}=(k_1,...,k_l)$ such that the following holds.

\smallskip

\noindent (i) The functions $f^\pm_0(w),g^\pm_0(w),f^\pm_1(w),g^\pm_1(w)$ are the $\mathbf{k}$-multisums of $f_0,g_0,f_1,g_1$ in the directions $d^\pm$, respectively; 

\smallskip

\noindent (ii) The holomorphic in  respectively $S^\pm$ functions $Y^\pm(w)$, constructed via $f^\pm_0(w),g^\pm_0(w),f^\pm_1(w),g^\pm_1(w)$ by using formulas \eqref{YY}, satisfy the ODE \eqref{merom}.

\end{corollary}
The last point is a consequence of uniqueness of multisummable functions. Since $Y$ is {\bf k}-multisommable on some multisectors, so are functions $w^mf'_0,wg'_0,w^mf'_1,w^mg'_1$. Thus equalities \eqref{YY} hold. We will also need that Corollary~\ref{summable} also holds with $f_0, f_1, g_0, g_1$
replaced by their conjugates, with the same multi-order $\mathbf{k}$:
\begin{corollary}\label{unnecessary}   There exist functions 
\begin{equation}\label{bared}
\overline{f^\pm_0}(w),\overline{g^\pm_0}(w),\overline{f^\pm_1}(w),\overline{g^\pm_1}(w),
\end{equation} 
holomorphic in the respective sectors $S^\pm$, which are the $\mathbf{k}$-multisums in the directions $d^\pm$ of  $\bar f_0(w),\bar g_0(w),\bar f_1(w),\bar g_1(w)$. Furthermore, the corresponding maps
$\overline{Y^\pm}$ defined as in \eqref{YY}, satisfy the meromorphic ODE $w^m \overline{Y^\pm} = \bar A(w,\overline{Y^\pm})$.
\end{corollary}   

The second main technical result of this section shows that $f$ and $g$ can be reproduced
from $f_0, f_1, g_0, g_1$. 
\begin{proposition}\label{ovcyann} There exist holomorphic functions $\varphi$ 
and $\psi$ defined in a neighbourhood of the origin in $\C^7$ such that the equality
\begin{equation}\label{Finally}
\begin{aligned}
f(z,w)=\varphi\bigl(z,w,g_0(w),wg_0'(w),f_0(w),f_1(w),g_1(w)\bigr), \\ 
g(z,w)=\psi\bigl(z,w,g_0(w),wg_0'(w),f_0(w),f_1(w),g_1(w)\bigr)
\end{aligned}
\end{equation}
holds for every $(f,g)$ as above. 
\end{proposition}

Before we turn to the proofs of the technical statements above, we show  how these statements imply \autoref{main} in the generic case.
\begin{proof}[Proof of \autoref{main} in the generic setting]
Let us introduce the functions
\begin{equation}\label{themap}
\begin{aligned}
f^\pm(z,w)=\varphi\bigl(z,w,g^\pm_0(w),w\cdot(g^\pm_0)'(w),f^\pm_0(w),f^\pm_1(w),g^\pm_1(w)\bigr), \\ 
g^\pm(z,w)=\psi\bigl(z,w,g^\pm_0(w),w\cdot(g^\pm_0)'(w),f^\pm_0(w),f^\pm_1(w),g^\pm_1(w)),
\end{aligned}
\end{equation}
well defined in the product of a disc $\Delta$ in $z$ centered at the origin and the sectors $S^\pm$ in $w$ (this product forms a {\em tangential sectorial domain}, as described in Section 2. $f^\pm(z,w),g^\pm(z,w)$  are asymptotically represented in their domains by $f(z,w),g(z,w)$ respectively, as follows from \eqref{Finally}. Based on \eqref{bared}, we similarly introduce $\overline{f^\pm}(z,w),\,\overline{g^\pm}(z,w),$ asymptotically representing $\bar f(z,w),\bar g(z,w)$, respectively.

Let us now consider the (complexified) basic identity
\begin{equation}\label{basicformal}
G(z,w)-\rho^*\bigl(F(z,w),\overline {F}(\xi,\eta),\overline {G}(\xi,\eta)\bigr)|_{w=\rho(z,\xi,\eta)}=0
\end{equation}
for the map $(F,G)$ between the germs at the origin of the initial hypersurfaces $M=\bigl\{w=\rho(z,\bar z,\bar w)\bigr\}$ and $M^*=\bigl\{w=\rho^*(z,\bar z,\bar w)\bigr\}$.
We claim that the sectorial map $(F^\pm(z,w),G^\pm(z,w))$ constructed via $f^\pm,g^\pm$ by the formula \eqref{normalmap} satisfies the basic identity \eqref{basicformal} as well, i.e.
\begin{equation}\label{basicCR}
G^\pm(z,w)-\rho^*\bigl(F^\pm(z,w),\overline {F^\pm}(\xi,\eta),\overline {G^\pm}(\xi,\eta)\bigr)|_{w=\rho(z,\xi,\eta)}=0, \quad (z,\xi,\eta)\in \Delta\times\Delta\times S^\pm.
\end{equation}

To prove the claim, let us analyze the identity \eqref{basicCR}. The left hand side of it, which we denote by 
$$\chi(z,\xi,\eta),$$
is holomorphic in $\Delta\times\Delta\times S^\pm$, respectively. Accordingly, the identity \eqref{basicCR} holds if and only if we have:
\begin{equation}\label{zxi}
\left.\frac{\partial^{p+q}}{\partial z^p\partial\xi^q}\chi(z,\xi,\eta)\right|_{z=\xi=0}\equiv 0, \quad p,q\geq 0.
\end{equation}
However, it is not difficult to verify (by applying the chain rule) that for each fixed $p,q\geq 0$ the left hand side in \eqref{zxi} is an analytic function $R_{p,q}$ in $\eta$, the sectorial functions
$$f^\pm_0(\eta),g^\pm_0(\eta),f^\pm_1(\eta),g^\pm_1(\eta),\overline{f^\pm_0}(\eta),\overline{g^\pm_0}(\eta),\overline{f^\pm_1}(\eta),\overline{g^\pm_1}(\eta),$$ 
and their derivatives of order $\leq p+q$. Hence, each left hand side in \eqref{zxi} is the $\mathbf{k}$-multisum of the identical analytic expressions $R_{p,q}$ in formal series, where $f^\pm_0(\eta),g^\pm_0(\eta),f^\pm_1(\eta),g^\pm_1(\eta)$ are replaced by the asymptotic expansions $f_0,g_0,f_1,g_1$, respectively, and  $\overline{f^\pm_0}(\eta),\overline{g^\pm_0}(\eta),\overline{f^\pm_1}(\eta),\overline{g^\pm_1}(\eta)$ by their asymptotic expansions $\bar f_0(\eta),\bar g_0(\eta),\bar f_1(\eta),\bar g_1(\eta)$, respectively. In view of the (valid!) formal basic identity \eqref{basicformal}, the latter formal series in $\eta$ vanish identically for any $p,q\geq 0$. The uniqueness property within the class of $\mathbf{k}$-multisummable series in the directions $d^\pm$ implies now that all the left hand sides in \eqref{zxi} {\em all vanish identically}.

As was explained in Section 2, the property \eqref{basicCR} for a sectorial map defined in a tangential sectorial domain implies that the restriction of the map onto the source manifold is a $C^\infty$ CR-map onto the target. Thus, the claim under discussion implies the assertion of the theorem. 
\end{proof}

The rest of this section is devoted to the proofs of the Propositions above. 

\subsection{Associated complete system} We show the following:
\begin{proposition}\label{associated}
Associated with a hypersurface \eqref{admissible} is  a second order singular holomorphic ODE $\mathcal E(M)$ given by 
\begin{equation}\label{ODE}
w''=w^m\Phi\left(z,w,\frac{w'}{w^m}\right),
\end{equation}   
where $\Phi(z,w,\zeta)$ is a holomorphic near the origin in $\CC{3}$ function with $\Phi=O(\zeta)$. The latter means that all Segre varieties of $M$ (besides the complex locus $X=\{w=0\}$ itself), considered as graphs $w=w_p(z)$, satisfy the ODE \eqref{ODE}. 
\end{proposition}
\begin{proof}
The argument of the proof very closely follows the one given in the proof of an analogues statement in \cite{nonminimalODE}, \cite{analytic} for the case of {\em $m$-admissible hypersurfaces}, and we leave the details of the proof to the reader.
\end{proof}

Based on the connection between mappings of hypersurfaces and that of the associated ODEs discussed in Section 2 and \autoref{specialmap}, we come to the consideration of ODEs \eqref{ODE} and formal power series mappings \eqref{normalmap} between them. We further recall that the fact that a mapping $(F(z,w),G(z,w))$ transforms an ODE $\mathcal E$ into an ODE $\mathcal E^*$ is equivalent to the fact  that the second jet prolongation $(F^{(2)},G^{(2)})$ transforms the ODEs $\mathcal E,\mathcal E^*$ into each other, where the ODEs are considered as submanifolds in $J^2(\CC{},\CC{})$. Applying this to two nonsingular ODEs $\mathcal E=\bigl\{w''=\Psi(z,w,w')\bigr\},\,  \mathcal E^*=\bigl\{w''=\Psi^*(z,w,w')\bigr\}$ and employing the classical jet prolongation formulas (e.g., \cite{bluman}), we obtain: 
\begin{multline}\label{trule}
\Psi(z,w,w')=\frac{1}{J}\left((F_z+w'F_w)^3\Psi^*\Bigl(F(z,w),G(z,w),\frac{G_z+w'G_w}{F_z+w'F_w}\Bigr)+\right.\\
+ I_0(z,w)+I_1(z,w)w'+I_2(z,w)(w')^2+I_3(z,w)(w')^3\Bigr),
\end{multline}
where $J:=F_z G_w-F_wG_z$ is the Jacobian determinant of the transformation and 
\begin{equation}\begin{aligned}
I_0 &=G_zF_{zz}-F_zG_{zz}\\
I_1 &=G_wF_{zz}-F_wG_{zz}-2F_zG_{zw}+2G_zF_{zw}\\
I_2 &=G_{z}F_{ww}-F_zG_{ww}-2F_wG_{zw}+2G_wF_{zw}\\
I_3 &=G_wF_{ww}-F_wG_{ww}.
\end{aligned}
\end{equation}
Setting then $\Psi(z,w,w'):=w^m\Phi\left(z,w,\frac{w'}{w^m}\right)$ (and similarly for $\Phi^*$) and switching to the notations in \eqref{normalmap}, we obtain the transformation rule for the class of ODEs \eqref{ODE} and mappings \eqref{normalmap} between them:

\begin{multline}\label{trule3}
w^m\Phi\left(z,w,\frac{w'}{w^m}\right)=\frac{1}{J}\Bigl[\bigl(1+ f_z+w' f_w)^3 (1+g_0(w)+w^{m-1}g\bigr)^m\cdot\\
\cdot w^m\Phi^*\Bigl(z+ f,w+wg_0(w)+ w^mg,\frac{w^mg_z+w'(1+ wg_0'+g_0+mw^{m-1}g+w^mg_w)}{w^m(1+g_0(w)+w^{m-1}g)^m(1+f_z+w' f_w)}\Bigr)+\\
+ I_0(z,w)+I_1(z,w)w'+I_2(z,w)(w')^2+I_3(z,w)(w')^3\Bigr],
\end{multline}
where
\begin{equation}\label{2jet}
\begin{aligned}
J&=(1+f_z)(1+g_0+wg_0'+w^mg_w+mw^{m-1}g)-w^mf_wg_z,\\
I_0 &= w^m\bigl(g_zf_{zz}- (1+f_z)g_{zz}\bigr),\\
I_1 &=\bigl(1+wg_0'+g_0+mw^{m-1}g+w^mg_w\bigr)f_{zz}-w^mf_wg_{zz}-\\
&-2(1+ f_z) 
(mw^{m-1}g_z+w^mg_{zw})+2w^m g_z f_{zw},\\
I_2 &=w^m g_{z}f_{ww}-(1+ f_z)(wg_0''+2g_0'+m(m-1)w^{m-2}g+2mw^{m-1}g_w+w^mg_{ww})-\\
&-2f_w(mw^{m-1}g_z+w^mg_{zw}) +2(1+wg_0'+g_0+mw^{m-1}g+w^mg_w) f_{zw},\\
I_3 &=(1+wg_0'+g_0+mw^{m-1}g+w^mg_w)f_{ww}-\\
&-f_w(wg_0''+2g_0'+m(m-1)w^{m-2}g+2mw^{m-1}g_w+w^mg_{ww}).
\end{aligned}
\end{equation}
Importantly, after putting $w'=\zeta w^m$, \eqref{trule3} becomes an identity of formal power series in the {\em independent} variables $z,w,\zeta$. 

We now extract from \eqref{2jet} four identities of power series {\em in $z,w$ only}, in the following way. For the first identity, we extract in \eqref{2jet} terms with $(w')^0$ and divide the resulting identity by $w^m$. For the second identity, we extract in \eqref{2jet} terms with $(w')^1$. For the third identity, we extract in \eqref{2jet} terms with $(w')^2$ and multiply the resulting identity (which has a pole in $w$ of order $m$) by $w^m$. For the last identity, we extract in \eqref{2jet} terms with $(w')^3$ and multiply the resulting identity (which has a pole in $w$ of order $2m$) by $w^{2m}$. The four resulting identities of formal power series in $z,w$ can be written as:
\begin{equation}\label{II}
\begin{aligned}
I_0=w^mT_0(z,w,j^1(f,g,g_0)),&\quad I_1=T_1(z,w,j^1(f,g,g_0)),\\ 
w^mI_2= T_2(z,w,j^1(f,g,g_0)),&\quad w^{2m}I_3=T_3(z,w,j^1(f,g,g_0)),
\end{aligned} 
\end{equation}
where $j^1(f,g.g_0)$ denotes the $1$-jet of $f,g,g_0$ (the collection of derivatives of order $\leq 1$), and $T_k(\cdot,z,w)$ are four precise holomorphic at the origin functions, exact form of which is of no interest to us.  We though emphasize two important properties of the identities \eqref{II}:

\smallskip

\noindent (a) the derivatives $f_w,g_w$ come in each $T_k$ with the factor $w^m$, and the derivative $g_0'$ comes in each $T_k$ with the factor $w$; 

\smallskip

\noindent (b) the derivatives $f_w,g_w,f_{zw},g_{zw}$ all come in all the left hand sides in \eqref{II} with the factor $w^m$, the derivatives $f_{ww},g_{ww}$ all come in all the left hand sides in \eqref{II} with the factor $w^{2m}$, and the derivatives $g'_{0},g''_0$ come in all the left hand sides in \eqref{II} with the factor $w$.

\smallskip

It is also not difficult to verify that the identities \eqref{II} are well defined, i.e. the formal power series under considerations all come into the right hand side in \eqref{II} with the zero constant term.

We can now prove Proposition~\ref{systemY}.
\begin{proof}[Proof of Proposition~\ref{systemY}]
We consider in the last two identities in \eqref{II} terms with $z^{0},z^{1}$, respectively. This gives us {\em four second order singular ODEs for the functions  $f_0,f_1,g_0,g_1$.} In the two identities with $z^0$, only the second order derivatives $f_0'',g_0''$ participate (the other derivatives have order $\leq 1$). It is not difficult to solve the latter identities for $w^{2m}f''_0,w^{m+1}g''_0$ (by applying the Cramer rule to the a nondegenerate linear system). We obtain, by combining the information in \eqref{2jet},\eqref{first} and the observations (a),(b) above:
\begin{equation}\label{solve2jet}
w^{2m}f''_0=U(y_1,y_2,...,y_8,w),\quad w^{m+1}g''_0=U(y_1,y_2,...,y_8,w),
\end{equation}
where $U$ and $V$ are two holomorphic at the origin functions in all their variables, exact form of which is of no interest to us. Using the $y$-notations and \eqref{first}, the equations \eqref{solve2jet} give:
\begin{equation}\label{second}
w^my_5'=\tilde U(y_1,y_2,...,y_8,w), \quad w^my_6'=\tilde V(y_1,y_2,...,y_8,w), 
\end{equation}
where, again, $\tilde U$ and $\tilde V$ are two holomorphic at the origin functions in all their variables, exact form of which is of no interest to us.

To obtain the missing conditions for $y_7',y_8'$, we use the system of two second order ODEs obtained by collecting in the last two identities of  \eqref{II} terms with $z^1$. Considering this system as a (nondegenerate) linear system in $w^{2m}f_1'',w^{2m}g_1''$ and solving by Cramer rule, we get:
\begin{equation}\label{solve2jet1}
w^{2m}f''_1=X(y_1,y_2,...,y_8,w^{2m}f''_0,w),\quad w^{2m}g''_1=Y(y_1,y_2,...,y_8,w^{m+1}g''_0,w),
\end{equation}
where $X$ and $Y$ are two holomorphic at the origin functions in all their variables, exact form of which is of no interest to us. Combining this with \eqref{solve2jet} and using \eqref{first}, we finally obtain
\begin{equation}\label{third}
w^my_7'=\tilde X(y_1,y_2,...,y_8,w), \quad w^my_8'=\tilde Y(y_1,y_2,...,y_8,w), 
\end{equation}
By putting \eqref{first},\eqref{second},\eqref{third} together, we obtain the Proposition.
\end{proof}

We now turn to the proof of Corollary~\ref{unnecessary}.

\begin{proof}[Proof of Corollary~\ref{unnecessary}]
We first want to show that the "barred" power series $\bar f_0(w),\bar g_0(w),\bar f_1(w),\bar g_1(w)$ belong to the same summability class as the original series. For doing so, let us consider the associated with \eqref{merom} ODE
\begin{equation}\label{merombar}
w^m\frac{dZ}{dw}=\bar A(w,Z),
\end{equation}
where $A(w,Y)$ is as in \eqref{merom}. We first note that the "barred" power series $\bar Y_0(w)$ satisfies the ODE \eqref{merombar}. Now, let us  write $\bold Y:=(Y,Z)$ and 
$$\bold A(w,\bold Y) :=\begin{pmatrix} A(w,Y) &0\\ 0 & \bar A(w,Z)\end{pmatrix},$$
and then consider the system
\begin{equation}\label{meromlong}
w^m\frac{d\bold Y}{dw}=\bold A(w,\bold Y).
\end{equation}
 Applying  \autoref{final-multisum} and \autoref{real-axis}  for the "decoupled"\, system \eqref{meromlong}, we find sectors $S^+,S^-\subset\CC{}$  containing the positive and the negative real lines (which we without loss of generality assume to be equal to the ones in \autoref{summable}), direction $d^\pm$ (which we without loss of generality assume to be equal to the ones in \autoref{summable}), and functions 
\begin{equation}
\overline{f^\pm_0}(w),\overline{g^\pm_0}(w),\overline{f^\pm_1}(w),\overline{g^\pm_1}(w),
\end{equation} 
holomorphic in the respective sectors $S^\pm$, which are the $\mathbf{k}$-multisums in the directions $d^\pm$ of  $\bar f_0(w),\bar g_0(w),\bar f_1(w),\bar g_1(w)$, respectively (we, again, assume without loss of generality that the multi-order $\mathbf{k}$ equals to the one in \autoref{summable}). In addition, the holomorphic in  respectively $S^\pm$ function $\overline{Y^\pm}(w)$, constructed via $\overline{f^\pm_0}(w),\overline{g^\pm_0}(w),\overline{f^\pm_1}(w),\overline{g^\pm_1}(w)$ by using formulas \eqref{YY}, satisfies the ODE \eqref{merom}. 
\end{proof}
 
 The last remaining piece is now the proof of Proposition~\ref{ovcyann}.

\begin{proof}[Proof of Proposition~\ref{ovcyann}] We now  consider the first two equations in \eqref{II}. Read together, they can be treated as a system of linear equations in $f_{zz},g_{zz}$ determinant of which at the origin is non-vanishing. Applying the Cramer rule, we obtain the following system of equations:
\begin{equation}\label{findfg}
f_{zz}=P(z,w,j^1(f,g),g_0,wg_0',f_{zw},g_{zw}), \quad g_{zz}=Q(z,w,j^1(f,g),g_0,wg_0',f_{zw},g_{zw}),
\end{equation}
where $P,Q$ are appropriate functions holomorphic in their arguments. We now consider the intimately related Cauchy problem
\begin{equation}\label{Cauchy}
f_{zz}=P(z,w,j^1(f,g),\alpha_0,\alpha_1,f_{zw},g_{zw}), \quad g_{zz}=Q(z,w,j^1(f,g),\alpha_0,\alpha_1,f_{zw},g_{zw})
\end{equation}
with the Cauchy data
\begin{equation}\label{Cdata}
f(0,w)=\beta_0,\quad f_z(0,w)=\beta_1,\quad g(0,w)=0,\quad g_z(0,w)=\beta_2,
\end{equation}
where $\alpha_i,\beta_j$ are additional parameters. By the parametric version of the Cauchy-Kowalevski theorem, namely the Ovcyannikov's theorem \cite{ovcyannikov, treves-ovcyannikov}, the latter Cauchy problem has a unique analytic solutions 
$$f=\varphi(z,w,\alpha_0,\alpha_1,\beta_0,\beta_1,\beta_2), \quad g=\psi(z,w,\alpha_0,\alpha_1,\beta_0,\beta_1,\beta_2),$$ where $\varphi$ and $\psi$ depend analytically on all their arguments. Hence, taking into account \eqref{normalmap},\eqref{expandfg}, we finally have the identities \eqref{Finally}. 
(we emphasize that the substitution of formal power series into $\varphi,\psi$ is well defined here, since {\em all} the formal data being substituted has no constant term!).
\end{proof}
\medskip

\section{The exceptional case}

In this section, we prove \autoref{main} in full generality. For that, we have to consider the case when, for an $m$-nonminimal at the origin hypersurface $M\subset\CC{2}$, the minimal part $M\setminus X$ contains Levi degenerate points. In this case, $M$ can {\em not} be associated to an ODE \eqref{ODE}. We overcome this difficulty by introducing {\em associated ODEs of high order.}

The proof of \autoref{main} in the general case has several ingredients, each of which we put in a separate subsection below. While we follow closely the structure of Section 3, the tools which 
we need to introduce in this section are considerably harder: We treat the multisummability of what will
later be ``initial terms'' in certain Cauchy problems in 4.1, then discuss the associated ODEs of 
higher order in 4.2. We can then prove \autoref{main} under an additional technical condition in 
4.3, and in 4.4 introduce the necessary geometrical concept to use this 
technical condition in full generality to complete the proof in 4.5.  

\subsection{$\mathbf{k}$-summability of initial terms}
In what follows, for hypersurfaces under consideration we consider the defining equation \eqref{mnonminimal}. Since $M$ is strictly pseudoconvex at generic points, a generic real-analytic curve 
$\Gamma\subset M$ through $0$, transverse to $T^c_0 M$, will not contain any Levi-degenerate point except for $0$. If we choose 
normal coordinates for which $\Gamma = \left\{ (z,w) \in M \colon z = 0 \right\}$ (which is possible, see e.g. \cite[Lemma 4.1]{lmblowups}, then its the complex defining equation 
\begin{equation}\label{compd}
w=\Theta(z,\bar z,\bar w) \quad \Theta(z,\bar z,\bar w)=\bar w + \sum_{j,k\geq 1} \Theta_{jk}(\bar w)z^k\bar z^l, \quad \Theta\not\equiv 0
\end{equation}
satisfies the additional condition 
\begin{equation}\label{Theta11}
\Theta_{11}(\bar w)\not\equiv 0.
\end{equation}
The fact that the minimal part $M\setminus X$ contains Levi degenerate points reads as
\begin{equation}\label{badcase}
\mbox{ord}_0\,\Theta_{11}(\bar w)>m.
\end{equation}

We start by considering for a formal power series map $(F,G)$ between germs at the origin of hypersurfaces \eqref{compd}  the expansion:
\begin{equation}\label{initmap}
F=\sum_{j\geq 0}F_j(w)z^j,\quad G=\sum_{j\geq 0}G_j(w)z^j.
\end{equation}
Arguing similarly to the proof of \autoref{specialmap}, it is not difficult to prove
\begin{lemma}\label{specialmap1}
The components of the formal map $(F,G)$ satisfy:
\begin{equation}\label{normalmap1}
F_z(0,0)=F_1(0)\neq 0, \quad G_w(0,0)=G_0'(0)\neq 0, \quad G(z,w)=O(w), \quad G_z(z,w)=O(w^{m}).
\end{equation}
Thus, in suitable coordinates,  we may assume
$$F_z(0,0)=F_1(0)=1,\quad G_w(0,0)=G_0'(0)=1.$$
\end{lemma}
Our goal in this subsection is to prove the following
\begin{proposition}\label{initterms}
There exist sectors $S^+,S^-\subset\CC{}$, containing the positive and the negative real lines respectively, directions $d^\pm\subset S^\pm$, a multi-order  $\mathbf{k}=(k_1,...,k_l)$, and functions $F^\pm_j(w),G^\pm_j(w)$ holomorphic in the respective sectors, such that for each $j\geq 0$, the functions $F^\pm_j(w),G^\pm_j(w)$ are the $\mathbf{k}$-multisums of $F_j,G_j$ in the directions $d^\pm$, respectively.
\end{proposition}
 \autoref{initterms} is proved in several steps.
 
\smallskip

\noindent{\bf Step I.} We first observe that the assertion of \autoref{initterms} is invariant under {\em biholomorphic} transformations of the  target. Indeed, a holomorphic coordinate change
$$z\mapsto U(z,w),\quad w\mapsto V(z,w)$$
in the target changes the components of the map as follows:
\begin{equation}\label{UV}
\widetilde F=U(F(z,w),G(z,w)),\quad\widetilde G=V((F(z,w),G(z,w)).
\end{equation}
The new coefficient functions $\widetilde F_j,\widetilde G_j$ can be computed by differentiating \eqref{UV} in $z$ sufficiently many times and evaluating at $z=0$, i.e. for some germs of holomorphic functions $C_{j_1 \dots j_r , \ell_1 \dots \ell_s} (z,w)$, $D_{j_1 \dots j_r , \ell_1 \dots \ell_s} (z,w)$ we can write  
\[\begin{aligned}
\widetilde F_j (w) &= \sum_{j_1  + \dots + \ell_s =j} C_{j_1 \dots j_r , \ell_1 \dots \ell_s}(F_0 (w), G_0 (w)) {F_{j_1} \dots F_{j_r } G_{\ell_1} \dots G_{\ell_s}}, \\
\widetilde G_j (w) &= \sum_{j_1  + \dots + \ell_s =j} D_{j_1 \dots j_r , \ell_1 \dots \ell_s}(F_0 (w), G_0 (w)) {F_{j_1} \dots F_{j_r } G_{\ell_1} \dots G_{\ell_s}}.
\end{aligned}
   \] 
   Thus the desired invariance property follows from an  application of Corollary~\ref{compo-multi}  (see the 
   discussion of the properties of multisummable functions in Section 2). 
\smallskip

\noindent{\bf Step II.} In this step, we make use of the following efficient blow-up procedure introduced in \cite{lmblowups} by Mir and the second author. 
\begin{lemma}[Blow-up Lemma, see \cite{lmblowups}]\label{blowuplemma}
Let $M\subset\CC{2}$ be a real-analytic hypersurface, which is Levi-degenerate at the origin and Levi-nonflat. Assume that  $M$ is given in coordinates \eqref{compd} and that the distinguished curve 
\begin{equation}
\Gamma=\{(z,w)\in M:\,\,z=0\}\subset M 
\end{equation}
does not contain Levi-degenerate points of $M$ other than the origin. Then there exists  a blow-down map 
\begin{equation}\label{lmblowup}
B(\xi,\eta):\quad (\CC{2},0)\lr (\CC{2},0),\quad B(\xi,\eta)=(\xi\eta^s,\eta),\quad s\in\mathbb{Z},\quad s\geq 2,
\end{equation}
and a real-analytic nonminimal at the origin hypersurface $M_B\subset\CC{2}_{(\xi,\eta)}$ with the complex locus $X=\{\eta=0\}$ such that:

\smallskip

(i) $B(M_B)\subset M,\quad B(X)=\{0\}$;

\smallskip

(ii) $M_B\setminus X$ is Levi-nondegenerate, and $M_B$ is given by an equation of the kind \eqref{admissible}.
\end{lemma}
We note at this point that the condition for $\Gamma$ in \autoref{blowuplemma} is precisely equivalent to \eqref{Theta11}.

We will need some control of the integer $s$ from the proof of the Blow-up Lemma. We quickly recall the needed details. For an $m$-nonminimal hypersurface,  transformations bringing to coordinates of the kind \eqref{compd} are associated with curves $\gamma\subset M$ passing through $0$ and transverse to the complex tangent at $0$. Such a curve $\gamma$ is being transformed into the distingusihed \eqref{Gamma} in the new coordinates \eqref{mnonminimal}. 

We then choose $\gamma$ in such a way that $\gamma\cap \Sigma=\{0\}$ for the Levi degeneracy set $\Sigma\subset M$, and bring to coordinates \eqref{compd}. This means that for the resulting hypersurface \eqref{compd} we have $\Theta_{11}\not\equiv 0$. For each $k\geq 2$, let us denote $$m(k):=\mbox{min}_{p+q=k}\,\mbox{ord}_0\Theta_{pq}.$$ 
We have $m(j)\geq m$ for all $j\geq 2$. 
After that, an integer $s$ in \eqref{lmblowup} is {determined} as any integer satisfying all the inequalities 
\begin{equation}\label{ineq}
2s+m(2)\leq ks+m(k),\quad k\geq 3.
\end{equation}
In fact, one can require the unique  (stronger) inequality
\begin{equation}\label{ineq1}
m(2)< s,
\end{equation}
and thus avoid considering $m(k),\,k\geq 3$.

We now proceed as follows.  We may assume that both $M$ and $M^*$ are given by coordinates \eqref{mnonminimal} with $\Theta_{11}\not\equiv 0$. We then fix an integer $s$, which satisfies \eqref{ineq1} for both $M$ and $M^*$. Next, we consider the formal curve $\gamma\subset M$ - the pre-image of \eqref{Gamma} under the given formal map $H=(F,G)$. Let us choose an {\em analytic} curve $\tilde \gamma\subset M$ tangent to $\gamma$ to order $s+1$, and a {\em biholomorphic} map $H_1$ transforming  $\tilde\gamma$ into \eqref{Gamma} and 
$M$ into a hypersurface $\widetilde M$ of the kind \eqref{mnonminimal}. Put $H_2:=H\circ H_1^{-1}$, so that $H=H_{2}\circ H_{1}$. Finally, put
$$\tilde \Gamma:=H(\tilde \gamma).$$
Note that, since $\gamma$ and $\tilde \gamma$ are tangent to order $s+1$, the same is true for \eqref{Gamma} and $\tilde\Gamma$.
 
We then can decompose $H^{-1}$ as a product 
\begin{equation}\label{HHH}
H^{-1}=H^{-1}_1\circ H^{-1}_2,
\end{equation}
where $H^{-1}_1$ is a {\em biholomorphic} map transforming  \eqref{Gamma} into $\tilde\gamma$  and $\widetilde M$ into $M$, and
$H^{-1}_2$ is a {\em formal} invertible map transforming  $\tilde \Gamma$ into \eqref{Gamma} and $M^*$ into the real-analytic hypersurface $\widetilde M$. Importantly, in view of the tangency condition, the formal map $H^{-1}_2$ satisfies 
\begin{equation}\label{goodord}
\mbox{ord}_0\,F_0(w)\geq s+1,
\end{equation}
where $F_0$ is as in \eqref{initmap}. Moreover, the blow up integer $s$ can be kept the same as before for the hypersurface $\widetilde M$ as well. Indeed, a transformation satisfying \eqref{goodord} clearly preserves the  corresponding integer $m^*(2)$ in \eqref{ineq1}(as we  choose $s>m^*(2)$), so that the inequalities \eqref{ineq1} still hold true for the same $s$ and the hypersurface $\widetilde M$.  

Finally, we recall that, in view of the considerations of Step I, the assertion of \autoref{initterms} applied for $H^{-1}_2$ is equivalent to that for $H^{-1}$. 

We summarize the considerations of Step II as follows: in view of the decomposition \eqref{HHH} and the subsequent properties of $H^{-1}_2$, 

\smallskip

{\em it is sufficient to prove \autoref{initterms} for maps $(F,G)$ satisfying, in addition, the inequality \eqref{goodord}}.   

\bigskip

\noindent{\bf Step III.} In this step, we are finally able to reduce \autoref{initterms} to the results already proved in the generic case. For that, we use the above blow up procedure.

In accordance with the outcome of the previous step, we consider a map $(F,G):\,(M,0)\mapsto (M^*,0)$ satisfying, in addition, \eqref{goodord}. Here the integer $s$ in \eqref{goodord} is an admissible integer for the blow down map \eqref{lmblowup} both in the source and in the target. After performing the blow ups (with the integer $s$ in \eqref{lmblowup}), we obtain real-analytic hypersurfaces $M_B, M^*_B$, respectively.

Re-calculating the map $(F,G)$ in the "blown up" coordinates $(\xi,\eta)$ gives:
\begin{equation}\label{newmap}
G_B(\xi,\eta)=G(\xi\eta^s,\eta), \quad F_B(\xi,\eta)=\frac{F_0(\eta)}{\eta^s}+F_1(\eta)\xi+\cdots,
\end{equation}
where dots stand for a power series in $\xi,\eta$ of the kind $O(\xi^2)$. In view of \eqref{goodord}, $F_B(\xi,\eta),\,G_B(\xi,\eta)$ are {\em well defined power series}. It is immediate then that the formal map 
$$H_B(\xi,\eta):=\bigl(F_B(\xi,\eta),\,G_B(\xi,\eta)\bigr)$$
transforms $(M_B,0)$ into $(M^*_B,0)$.
Furthermore, in view of \eqref{normalmap1}, the formal map $H_B(\xi,\eta)$ is {\em invertible}, so that the results of Section 3 are applicable to it. Expanding now
$$F_B(\xi,\eta)=\sum_{j\geq 0}F^B_j(\eta)\xi^j,\quad  
G_B(\xi,\eta)=\sum_{j\geq 0}G^B_j(\eta)\xi^j,$$
and applying to $H_B$ the assertion of \autoref{summable} and the formulas \eqref{Finally}, we immediately obtain for the components $F^B_j,G^B_j$ the desired $\mathbf{k}$-summability property (identical to the one stated in \autoref{initterms}). At the same time, the relations \eqref{newmap} show that
\begin{equation}\label{connecting}
G^B_j(\eta)=\eta^{sj}G_j(\eta).
\end{equation}
We immediately obtain from \eqref{connecting} the assertion of \autoref{initterms} for the components $G_j$ (with the same sectors, multi-directions and multi-order $\mathbf{k}$ as for $F_B,G_B$). Finally, since we have
$$ F(\xi\eta^s,\eta)=\bigl(G(\xi\eta^2,\eta)\bigr)^s\cdot F_B(\xi,\eta),$$
the chain rule and the multisummability property for $F^B_j,G_j$ imply the assertion of \autoref{initterms} for the components $F_j$. This finally proves \autoref{initterms}.

\qed

\subsection{Associated ODEs of high order} In this section we consider the case when the source and the target $m$-nonminimal hypersurfaces satisfy the additional {\em $k$-nondegeneracy} condition. The latter means that for some $k\geq 1$ we have 
\begin{equation}\label{knondeg}
\mbox{ord}_0\,\Theta_{k1}=m
\end{equation}
for the defining function \eqref{compd}. As a well known fact (e.g. \cite{meylan}) the property of being $m$-nonminimal $k$-nondegenerate is  invariant under (formal) invertible transformations. In view of \eqref{badcase}, we may assume that $k\geq 2$ in our setting.

The main goal of this section is to show that we can associate  a {\em system $\mathcal E(M)$ of  $k$ singular ODEs of orders $\leq k+1$} to  an $m$-nonminimal $k$-nondegenerate hypersurface $M$.  By the latter we mean (as in the generic case) that all the Segre varieties $Q_p$ of $M$ for $p\not\in X$ considered as graphs $w=w_p(z)$ satisfy the system of ODEs $\mathcal E(M)$ as follows: 

\begin{proposition}
Let $M\subset \C^2 $ be an $m$-nonminimal $k$-nondegenerate hypersurface. Then there exists a 
system of holomorphic ODEs $\mathcal E (M)$, called the associated system to $M$,  of the form 
 \begin{equation}\label{assocd}
w'=\Phi_1\left(z,w,\frac{w^{(k)}}{w^m}\right),\cdots,w^{(k-1)}=\Phi_{k-1}\left(z,w,\frac{w^{(k)}}{w^m}\right),\,\,\,w^{(k+1)}=\Phi\left(z,w,\frac{w^{(k)}}{w^m}\right)
\end{equation}
such that  $w = w_p (z)$ is a solution of $\mathcal E (M)$ for $p\notin X$ and 
\begin{equation}\label{factord}
\Phi_j = O(w^m\zeta), j=1, \dots, k .
\end{equation} Furthermore, we 
can assume that $\Phi_1 (z,w,\zeta)$ satisfies
\[ {\text{Property }} (*): \quad \dopt{^{k+s}\Phi_1}{w^m \zeta^k} (0) =k!s! (\pm i) . \]
 If $M^*$ is another such hypersurface, than 
 any formal map $(F,G)$ taking $M$ 
 into $M^*$ satisfies \eqref{normalmap1} and  transforms the system 
$\mathcal E(M)$ into $\mathcal E(M^*)$. 
\end{proposition}

For producing the associated ODEs, we consider the Segre family of an $m$-nonminimal hypersurface \eqref{compd} satisfying the additional $k$-nondegeneracy condition, and produce for it an elimination procedure, in the spirit of that discussed in Section 2. This Segre family looks as:
\begin{equation}\label{segrefam}
w=b+ O(ab^mz)
\end{equation}
(we use the notation $p=(\bar a,\bar b)$). Differentiating \eqref{segrefam} $k$ times in $z$  and using \eqref{knondeg}, we obtain:
\begin{equation}\label{segredif}
w^{(k)}=ab^m(\alpha +o(1)), \quad \alpha\neq 0
\end{equation}
(here $\alpha$ is a fixed constant).
Dividing \eqref{segredif} by the $m$-th power of \eqref{segrefam} gives:
\begin{equation}\label{segredif1}
\frac{w^{(k)}}{w^m}=\alpha a+o(a).
\end{equation}
Solving the system \eqref{segredif1},\eqref{segrefam} for $a,b$ by the implicit function theorem  yields
\begin{equation}\label{solved}
a=A\left(z,w,\frac{w^{(k)}}{w^m}\right), \quad b=B\left(z,w,\frac{w^{(k)}}{w^m}\right)
\end{equation}
for two holomorphic near the origin in $\CC{3}$ functions $A(z,w,\zeta), B(z,w,\zeta)$ with $A=O(\zeta)$ and $B= O(w)$. Differentiating then \eqref{segrefam} $j$ times for each $j=1,...,k-1,k+1$ and substituting  \eqref{solved} into the results finally gives us \eqref{assocd}. 
Note that $\Phi_1(z,w,\zeta),...,\Phi_{k-1}(z,w,\zeta),\Phi(z,w,\zeta)$ are all holomorphic near the origin in $\CC{3}$ functions which satisfy \eqref{factord}
(as follows from the elimination procedure). 
It is immediate, in the same way as in the nondegenerate case, that {\em all the Segre varieties $Q_p$ of $M$ for $p\not\in X$ considered as graphs $w=w_p(z)$ satisfy the system of ODEs $\mathcal E(M)$.}

We now turn to property (*) for the ODE system \eqref{assocd}. For obtaining it, let us recall that defining equations \eqref{compd} of hypersurfaces under consideration satisfy the reality condition:
 \begin{equation}\label{reality}
w\equiv \Theta(z,\bar z,\bar\Theta(\bar z,z,w))\,\,\, \forall z,\bar z,w
\end{equation} 
(see, e.g., \cite{ber}). Gathering in \eqref{reality} terms with $z^k\bar z^1$ and using \eqref{compd}, we obtain
$$0=\Theta_{k1}(w)+\bar\Theta_{1k}(w).$$
Hence we have, in view of \eqref{knondeg}:
\begin{equation}\label{1k}
\mbox{ord}_0\,\Theta_{1k}=m.
\end{equation}
It immediately follows then from the above elimination procedure that 
the term with $z^0w^m\zeta^{k}$ in the expansion of the function $\Phi_1$ in \eqref{assocd} is non-zero,  and without loss of generality, we assume its coefficient in what follows to be equal to $\pm i$ (even though its exact value
is of no special interest to us) and hence Property (*) holds.  





\smallskip

\subsection{Proof of the main theorem under the $k$-nondegeneracy assumption}

The main technical difficulty of this
 subsection is to provide an analogue of \eqref{findfg} for mappings
between hypersurfaces which satisfy the $ m $-nonminimal $ k $-nondegeneracy assumptions. 
We will   expand such a map $ H=(F,G) $ in this section as 
\begin{equation}\label{expd}
\begin{aligned}
F&=z+S(z,w)+\sum_{j=0}^k f_j(w)z^j+f(z,w),& \quad \\
G&=T(w)+w^mR(z,w)+wg_0(w)+w^m\sum_{j=1}^k g_j(w)z^j+w^mg(z,w), \\
S_z&(0,0)=0, \quad f(z,w)=O(z^{k+1}),\quad g(z,w)=O(z^{k+1}),
\end{aligned}
 \end{equation}
where $f_j,g_j,f,g$ are formal power  series, $f_j,g_j$ all vanish to order $k+1$, and $T(w),S(z,w),R(z,w)$ are certain  { \em fixed polynomials} in their variables, exact form of which is of no interest to us (the desired representation of $g$ is possible in view of \eqref{normalmap1}). We think about $f_j,g_j$ and their derivatives as ``additional parameters'', and for this purpose, we write  
$$\alpha_{ij}:=f_i^{(j)}(w), \quad \beta_{ij}:=g_i^{(j)}(w), \quad \bold\alpha=\{\alpha_{ij}\}, \quad  \bold\beta=\{\beta_{ij}\}, \quad 0\leq i\leq k, \quad 0\leq j\leq k+1.$$
 We 
will see (through a careful analysis of the transformation rules for the 
associated systems) that we can find holomorphic functions $ U $, $ V $, such that  
\begin{equation}\label{newcauchy}
\begin{aligned}
f_{z^{k+1}}=U\Bigl(z,w,j^k(f,g),\{f_{z^{k+1-j}w^{j}}\}_{j=1}^{k+1},\{g_{z^{k+1-j}w^{j}}\}_{j=1}^{k+1},
\alpha,\beta \Bigr),\\
g_{z^{k+1}}=V\Bigl(z,w,j^k(f,g),\{f_{z^{k+1-j}w^{j}}\}_{j=1}^{k+1},\{g_{z^{k+1-j}w^{j}}\}_{j=1}^{k+1},
\alpha,\beta \Bigr).
\end{aligned}
\end{equation}

Let us first show how \autoref{main} follows (in the $k$-nondegenerate case) from \eqref{newcauchy}.

\begin{proof}[Proof of \autoref{main} under the $k$-nondegeneracy assumption]

Solving \eqref{newcauchy} by the Ovcyannikov theorem (see Section 3) as a Cauchy problem with the initial data
$$f_{z^j}(0,w)=g_{{z}^j}(0,w)=0, \quad 0\leq j\leq k$$
and the additional parameters $\alpha,\beta$, 
we obtain:
\begin{equation}\label{finally1}
\begin{aligned}
f(z,w)=\varphi\bigl(z,w,\alpha,\beta\bigr), \quad g(z,w)=\psi\bigl(z,w,\alpha,\beta\bigr)
\end{aligned}
\end{equation}
for two functions $\varphi,\psi$, holomorphic in all their arguments. We recall now that, by definition, $\alpha$ and $\beta$ stand for formal power series {\em without constant term}, so that substituting back $f_i^{(j)}(w)$ for $\alpha_{ij}$ and $g_i^{(j)}(w)$ for $\beta_{ij}$ is well defined. 

Let us note finally that, combining \autoref{initterms} and the expansion \eqref{expd}, we may apply the assertion of \autoref{initterms} to the functions $f_0,...,f_k,g_0,...,g_k$.  Substituting the arising sectorial functions $f^\pm_j,g^\pm_j$ into \eqref{finally1}, we obtain sectorial holomorphic transformations $(F^{\pm},G^{\pm})$. Then, arguing identically to the proof of \autoref{main} in the generic case (end of Section 3), we obtain the assertion of \autoref{main} in the $k$-nondegenerate case.
\end{proof}

In what follows, we have to take into consideration the space $J^{k+1}(\CC{},\CC{})$ of $(k+1)$-jets of holomorphic maps from $\CC{}$ into itself. We use the notations
$$(z,w,w_1,...,w_{k+1})$$
for the coordinates in the jet space (here $w_j$ corresponds to the derivative $w^{(j)}(z)$). A system 
\eqref{assocd} shall be regarded then as a submanifold in $J^{k+1}(\CC{},\CC{})$ of dimension $3$ (with the local coordinates $z,w,w_k$):
\begin{equation}\label{submd}
w_1=\Phi_1\left(z,w,\frac{w_k}{w^m}\right),\cdots,w_{k-1}=\Phi_{k-1}\left(z,w,\frac{w_k}{w^m}\right),\,\,\,w_{k+1}=\Phi\left(z,w,\frac{w_k}{w^m}\right).
\end{equation}

Next, we consider the $(k+1)$-jet prolongation 
\begin{multline}
H^{(k+1)}(z,w,w_1,...,w_{k+1})=\\
=\bigl(F(z,w),G(z,w),G^{(1)}(z,w,w_1),G^{(2)}(z,w,w_1,w_2),...,G^{(k+1)}(z,w,w_1,...,w_{k+1})\bigr)
\end{multline}
of the map $H=(F,G)$. 
Introducing the total derivation operator
\begin{equation}\label{DD}
D:=\partial_z+w_1\partial_w+\sum_{j\geq 1} w_{j+1}\partial_{w_j},
\end{equation}
we can inductively compute the components of the prolonged map (see  \cite{bluman}[(3.96d) of section 2.3.1]) as
\begin{equation}\label{findGk}
G^{(j)}=\frac{DG^{(j-1)}}{DF},\,\,j\geq 1,\quad\mbox{where}\quad G^{(0)}:=G.
\end{equation}
Note that, in fact, $$DF=F_z+w_1G_w.$$ 
As follows from \eqref{findGk}, each $G^{(j)}(z,w,w_1,...,w_j)$ is an expression, {\em rational} in the first jet variable $w_1$ and {\em polynomial} in the remaining jet variables $w_2,...,w_j$; its coefficients are {\em universal polynomials} in the $j$-jet of $(F,G)$. For certain precise values of $j$ (e.g. $j=1,2,3$), the $j$-jet prolongation formulas can be written explicitly. For example, we have:
\begin{equation}\label{G23}
\begin{aligned}
G^{(1)}(z,w,w_1)&=\frac{G_z+w_1G_w}{F_z+w_1F_w},\\
G^{(2)}(z,w,w_1,w_2)&=\frac{1}{(F_z+w_1F_w)^3}\Bigl[(F_z+w_1F_w)(G_{zz}+2w_1G_{zw}+(w_1)^2G_{ww}+w_2G_w)-\\
&-(G_z+w_1G_w)(F_{zz}+2w_1F_{zw}+(w_1)^2F_{ww}+w_2F_w)
\Bigr].
\end{aligned}
\end{equation}
For some higher orders see, e.g., \cite{bluman}. However, for a general $j$, only  certain summation formulas exist, which can not always be worked out. That is why we will use only a few properties of the prolonged maps, which are useful for our consideration. For example, we can claim that, for maps of the kind \eqref{normalmap1}, the denominator of it is non-vanishing at $z=w=w_1=...=w_j=0$. This can be easily proved by induction, by using \eqref{findGk}  and the fact that $F_z(0,0)=1$.

According to the outcome of the previous section and the discussion in Section 2, the prolonged map $H^{(k+1)}$ transforms  the submanifolds $\mathcal E(M),\mathcal E(M^*) \subset J^{k+1}(\CC{},\CC{})$ into each other. That is, we have the following basic identity (we set $W_k:=(w_1,...,w_k)$):
\begin{equation}\label{tangency}
\begin{aligned}
 G^{(1)}(z,w,w_1)&=\Phi^*_1\left(F(z,w),G(z,w),\frac{G^{(k)}(z,w,W_k)}{G^m(z,w)}\right),\\
&\cdots\\
 G^{(k-1)}(z,w,w_1,\ldots,w_{k-1})&=\Phi^*_{k-1}\left(F(z,w),G(z,w),\frac{G^{(k)}(z,w,W_k)}{G^m(z,w)}\right),\\
 G^{(k+1)}(z,w,w_1,\ldots,w_{k+1})&=\Phi^*\left(F(z,w),G(z,w),\frac{G^{(k)}(z,w,W_k)}{G^m(z,w)}\right), \\
\end{aligned}
\end{equation}
subject to the restriction 
\begin{equation}\label{restrd}
  w_1=\Phi_1\left(z,w,\frac{w_k}{w^m}\right),\cdots,w_{k-1}=\Phi_{k-1}\left(z,w,\frac{w_k}{w^m}\right),\,\,\,w_{k+1}=\Phi\left(z,w,\frac{w_k}{w^m}\right)
 \end{equation}
(here we used the star notation for the target ODE system).
We claim that, by setting 
\begin{equation}\label{zeta}
\zeta:= \frac{w_k}{w^m},
\end{equation} 
we can understand \eqref{tangency} as an identity of formal power series in the {\em independent} variables $z,w,\zeta$. 

Indeed, we first note that the substitution $w_k=w^m\zeta$ makes all expressions in \eqref{restrd} power series in $z,w,\zeta$ (divisible by $\zeta$, in view of \eqref{factord}). Further, we
consider the singular expression $\frac{G^{(k)}(z,w,w_1,..,w_k)}{G^m(z,w)}$ in \eqref{tangency} as a ratio of two formal power series $P(z,w,w_1,...,w_k),Q(z,w,w_1,...,w_k)$, each of which is polynomial in $w_1,...,w_k$. The denominator $Q$ can be factorized as $w^m\cdot \widetilde Q(z,w)$ with $\widetilde Q(0,0)\neq 0$ (as follows from \eqref{normalmap}).
Next,  the ``constant'' term of the polynomial $P$ obtained by setting $w_j=0$ for all $j$, can be inductively computed using the scheme
\begin{equation}\label{cterm}
c_1=\frac{G_z}{F_z},\quad c_j=\frac{\partial_z (c_{j-1})}{F_z},\quad  2\leq j \leq k+1
\end{equation}
(as follows from \eqref{findGk}), and it follows then from \eqref{normalmap1} that the desired constant term $c_k(z,w)$ is divisible by $w^m$. All the other terms in $P$ are (i) either divisible by $w_k$, hence the substitution $w_k=w^m\zeta$ makes them divisible by $w^m$, or (ii) divisible by some $w_j$, $j=1,...,k-1$, and hence the substitution $w_j=\Phi_j(z,w,w_1,...,w_j)$ makes them divisible by $w^m$ (in view of \eqref{factord}). We conclude that $P$ subject to the restriction \eqref{restrd} is divisible by $w^m$ (after  the substitution $w_k=w^m\zeta$), and this proves the claim.


We use the following

\smallskip

\noindent{\bf Convention.} In what follows, 
$$h_{z^kw^l}$$
denotes the partial derivative $\frac{\partial^{k+l}}{\partial z^k\partial z^l}$ for a function $h(z,w)$.

\smallskip

We then consider the last equation in \eqref{tangency} subject to \eqref{restrd} as an identity in $z,w,\zeta$ and collect within it all terms with $\zeta^0$. Then:

\smallskip

(i) in the left hand side, we obtain the expression $c_{k+1}(z,w)$ from \eqref{cterm}; it is easy to see that this expression can be written as 
$$\frac{1}{(F_z)^{k+1}}\bigl(G_{z^{k+1}}\cdot F_z-F_{z^{k+1}}\cdot G_z+\cdots\bigr),$$
where dots stand for a polynomial in $F_z,F_{zz},...,F_{z^k},G_z,g_{zz},...,G_{z^k}$; substituting \eqref{expd}, we obtain 
$$w^m\bigl(g_{z^{k+1}}\cdot(1+A)+f_{z^{k+1}}\cdot B+C\bigr),$$
where $A,B,C$ are {\em holomorphic} expressions in $j^kf,j^kg,z,w,\alpha,\beta$, and $A,B$ vanish at the origin. (In fact, $A,B,C$ have more specific form, but we do not need these further details);

\smallskip

(ii) for the right hand side, we argue as above and conclude that, for the singular argument  $\frac{G^{(k)}(z,w,w_1,..,w_k)}{G^m(z,w)}$, evaluating $\zeta=0$ and  substituting \eqref{expd} makes the numerator divisible by $w^m$. Taking further \eqref{factord} into account, we conclude that the right hand side in the identity under consideration as well has the form 
$$w^m\widetilde C,$$
where $\widetilde C$ is an  expression, {\em holomorphic}  in $j^kf,j^kg,z,w,\alpha,\beta$. 

We summarize that, gathering in the last identity in \eqref{tangency} terms with $\zeta^0$ gives:
\begin{equation}\label{firsteq}
g_{z^{k+1}}\cdot (1+A)+f_{z^{k+1}}\cdot B=\widehat C,
\end{equation}
where $A,B,\widehat C$ are holomorphic expressions as above, and $A,B$ vanish at the origin.

It remain for us to obtain one more identity of the kind \eqref{firsteq}, solvable already in $f_{z^{k+1}}$. For doing so, we consider in the last identity in \eqref{tangency} (subject to restriction \eqref{restrd}) terms with $\zeta^k$. 

\medskip

\noindent{\bf Claim.} The result of collecting terms with $\zeta^k$ in \eqref{tangency} (subject to restriction \eqref{restrd})
can be written in the form
\begin{equation}\label{secondeq}
-f_{z^{k+1}}\cdot\bigl(\pm i +L_0\bigr)+\sum_{j=1}^{k+1} L_j\cdot f_{z^{k+1-j}w^{j}}+\sum_{j=0}^{k+1} M_j\cdot g_{z^{k+1-j}w^{j}}+N=\widetilde N,
\end{equation}
where  the expressions $L_j,M_j,N,\widetilde N$ are described identically to the expressions $A,B,\widehat C$ in \eqref{firsteq} and, moreover,  $L_0$ vanishes at the origin.

\medskip

To prove the claim, we have to analyze the jet prolonged component $G^{(k+1)}$ with more details. Recall that $G^{(k+1)}$ is a rational in $w_1$ and polynomial in $w_2,...,w_{k+1}$ expression, coefficients of which are certain universal polynomials in $j^{k+1}(F,G)$. Its denominator is nonvanishing for $z=w=w_1=...=w_{k+1}$, as discussed above. Hence, we may expand
 \begin{equation}\label{expandGk}
G^{(k+1)}=\sum_{l_1,...,l_{k+1}\geq 0} E_{l_1,...,l_{k+1}}(w_1)^{l_1}\cdots(w_{k+1})^{l_{k+1}},
\end{equation}
where $E_{l_1,...,l_{k+1}}$ are all certain universal polynomials in $j^{k+1}(F,G)$ and the ratio $\frac{1}{F_z}$ (the latter fact can be seen from \eqref{findGk}, induction in $k$ and the chain rule). Recall that the constant term $E_{0,...,0}$ can be computed via \eqref{cterm}. For all the other terms, we have to distinguish $E_{l_1,...,l_{k+1}}$ depending on the {\em highest order derivatives $F_{z^{p}w^{q}},G_{z^{p}w^{q}},\,p+q=k+1$. In this regard, we have
\begin{lemma}\label{highord}
The only coefficients $E_{l_1,...,l_{k+1}}$ in \eqref{expandGk} depending on the  highest order derivatives $F_{z^{p}w^{q}},G_{z^{p}w^{q}},\,p+q=k+1$ are $E_{l,0,...,0},\, l\geq 0$. Moreover, $E_{1,0,...,0}$ has the form identical to the left hand side of
\eqref{secondeq}, where the expressions $L_j,M_j,N$ are certain universal polynomials in $j^k(F,G)$ and $\frac{1}{F_z}$, 
and, in addition, $L_0$ vanishes when $j^k(F,G)$ is evaluated at $z=w=0$.
\end{lemma}
\begin{proof}
For $k=0$, the assertion follows from the formula
\eqref{G23}
and the chain rule. For  $k>0$, we apply the iterative formula \eqref{findGk}, induction in $k$ and the chain rule. 
Then the assertion of the lemma follows by a straightforward inspection. 
\end{proof}
\rm
 We immediately conclude that, when collecting terms with $\zeta^k$ in the left hand side of the last identity in \eqref{tangency}, highest order derivatives may arise only from terms with $(w_1)^l,\,l\geq 1$. Next, we note that the term $E_{1,0,...,0}\cdot w_1$, subject to constraint \eqref{restrd}, contributes 
$$w^m(\pm i +o(1))\cdot E_{1,0,...,0}$$
(as follows  from the Property $(*)$ of $\Phi_1$). Substituting the expansions \eqref{expd} for $F,G$, we obtain an expression of the kind \eqref{secondeq} multiplied by $w^m$. Further, the constant term $E_{0,...,0}$ does not contribute to $\zeta^k$ (as it doesn't depend on the $w_j$'s). All terms with $(w_1)^l,\,l\geq 2$ may contribute to $\zeta^k$, however, in view of the factorization property \eqref{factord} their contribution gives at least the factor $O(w^{2m})$ in front of the highest order derivatives. All other terms do {\em not} contribute to $\zeta^k$ with the highest order derivatives, as follows from \autoref{highord}. They, however, necessarily give the factor $w^m$, in view of \eqref{factord},\eqref{zeta}.

We finally conclude that the left hand side of the identity under discussion has the form identical to the left hand side in  \eqref{secondeq} multiplied by $w^m$.

To study the right hand side of the last identity in \eqref{tangency} subject to \eqref{restrd}, we recall that 

\smallskip

(i) the denominator of the singular argument $\frac{G^{(k)}(z,w,w_1,..,w_k)}{G^m(z,w)}$ has the form $w^m\cdot\widetilde G(z,w)$, $\widetilde G(0,0)\neq 0$; 

\smallskip

(ii) the constant term in the numerator of the same expression is divisible by $w^m$, after substituting \eqref{expd} (as discussed above); 

\smallskip

(iii) the substitutions \eqref{zeta},\eqref{restrd} together with the factorization \eqref{factord} makes the rest of the numerator divisible by $w^m$; 

\smallskip

(iv) the factorization \eqref{factord} applied to $\Phi^*$ makes the right hand side under consideration in addition divisible by  $w^m$ (after substituting \eqref{expd}). 

In this way, we conclude that the right  hand side of the identity under discussion has the form identical to the right hand side in  \eqref{secondeq} multiplied by $w^m$. Dividing the latter identity by $w^m$ finally proves the claim. \qed

We can now consider the identities \eqref{firsteq},\eqref{secondeq} as a linear system in $f_{z^{k+1}},g_{z^{k+1}}$. Solving it 
by the Cramer rule, we finally obtain \eqref{newcauchy}. 

\qed

\subsection{Pure order of a nonminimal hypersurface}

It might still happen that an $m$-nonminimal hypersurface \eqref{compd} does {\em not} satisfy the $k$-nondegeneracy assumption. To deal with this case, we do (in appropriate coordinates) a blow up in the space of parameters of the Segre family. Related to this procedure is an important invariant of real hypersurface which we  call the {\em pure order}. From the point of view of our method, the pure order replaces, in a certain sense, the nonminimality order.
\begin{definition}\label{pureord}
Let $M\subset\CC{2}$ be a real-analytic Levi-nonflat hypersurface given by \eqref{compd}. The {\em pure order} of $M$ at $0$ is the integer $p$ such that
\begin{equation}\label{pure}
p+1=\mbox{min}_{k,l\geq 1}\,\{l+\mbox{ord}_0\,\Theta_{kl}(\bar w)\}.
\end{equation}
In other words, $p+1$ is the the minimal possible $l+s$ such that for some $k>0$ the term with $z^k\bar z^l\bar w^s$ in the expansion of $\Theta$ does not vanish. 
\end{definition}
Note that:

\smallskip

(i) for a Levi-nonflat hypersurface $p$ is well defined and nonnegative;

(ii) for a Levi-nondegenerate hypersurface we have $p=0$;

\smallskip

(iii) for an $m$-nonminimal hypersurface we have $p\geq m$;

\smallskip

(iv) for an $m$-nonminimal hypersurface with $M\setminus X$ Levi-nondegenerate (the generic case from Section 3) we have $p=m$.

We start with showing that the integer $p$ is a (formal) invariant of a real-analytic hypersurface.

\begin{proposition}
The pure order of a Levi-nonflat hypersurface is invariant under (formal) invertible transformations of hypersurfaces \eqref{compd}. 
\end{proposition}

\begin{proof}
We note that the pure type introduced above actually comes from an {\em invariant pair}
as introduced in \cite{Ebenfelt:2009uh}; the invariance of those is 
proved in that paper. 
\end{proof}

We now apply the notion of the pure type to prove the following factorization property for CR-maps.
\begin{proposition}\label{factord2}
Let $M,M^*\subset\CC{2}$ be two real-analytic nonminimal at the origin hypersurfaces, and $H=(F,G):\,(M,0)\mapsto (M^*,0)$ a formal map. Then  
\begin{equation}\label{normalmap2}
F_z(0,0)\neq 0, \quad G_w(0,0)=G_0'(0)\neq 0, \quad G(z,w)=O(w), \quad G_z(z,w)=O(w^{p+1}),
\end{equation}
\end{proposition}
where $p$ is the pure order of $M,M^*$ at $0$.
\begin{proof}
The proof of all the assertions except the last one goes identically to the proof of \eqref{normalmap}. For the property $G_z(z,w)=O(w^{p+1}),$ we consider the basic identity 
\begin{equation}\label{basic2}
G(z,w)=\Theta^*(F(z,w),\bar F(\bar z,\bar w), \bar G(\bar z,\bar w))|_{w=\Theta(z,\bar z,\bar w)}.
\end{equation}
Putting $\bar z=0$, we get $w=\bar w$, and further differentiating in $z$ gives:
\begin{equation}\label{basic3}
G_z(z,\bar w)=\frac{\partial}{\partial z}\Bigl[\Theta^*(F(z,\bar w),\bar F(0,\bar w), \bar G(0,\bar w))\Bigr].
\end{equation}
We note now that every non-zero term in the expansion of $\Theta^*$ in $z,\bar z,\bar w$ has total degree at least $p+1$ in $\bar  z,\bar w$ (by the definition of $p$). At the same time, since $(F,G)$ preserves the origin, we have $g(0,\bar w)=O(\bar w),\, F(0,\bar w)=O(\bar w)$, so that the whole expression in the square brackets in \eqref{basic3} becomes divisible by $\bar w^{p+1}$. This property persists after differentiating in $z$, and this proves the proposition.
\end{proof}
Next, we prove
\begin{proposition}\label{p1}
Let $M\subset\CC{2}$ be a real-analytic Levi-nonflat hypersurface, and $p$ is its pure type at $0$. Then there exist local holomorphic coordinates \eqref{compd} for $M$ at the origin with \eqref{Theta11}, such that for certain $k\geq 1$ we have:
\begin{equation}\label{p1term}
 \mbox{ord}_0\,\Theta_{k1}(\bar w)=p.
\end{equation} 
\end{proposition}
\begin{proof}
Let us choose any coordinates \eqref{compd} for $M$ at $0$ with \eqref{Theta11}. As was discussed above, change of coordinates \eqref{compd} corresponds to chossing a curve $\gamma\subset M$ being transformed to the canonical curve \eqref{Gamma}. Let us choose $\gamma\subset M$ of the kind
$$z=\alpha u, \quad w=u+iq(u), \quad u\in\RR{}$$
for an appropriate real-valued $q(u)$ and a generic $\alpha\in\CC{}$. Then there exists a biholomorphic transformation of the form,
\begin{equation}\label{curve}
z\mapsto z-\alpha w, \quad w\mapsto g(z,w),  \quad g(0,0)=0,
\end{equation}
mapping $M$ into another hypersurface $M^*$ of the form \eqref{compd} and $\gamma$ into the curve \eqref{Gamma} (e.g., \cite{chern}\cite{lmblowups}). If we now fix in the expansion \eqref{compd} of $M$  the non-zero term $z^k\bar z^j\bar w^l,\,j+l=p+1$ with the minimal
 $k\geq 1$, then it is easy to verify from the basic identity that the substitution \eqref{curve} creates, for a  generic $\alpha$, a non-zero term $z^k\bar z\bar w^p$ in the expansion \eqref{compd} for $M^*$. In view of the invariance of the pure order this means the validity of \eqref{p1term} for $M^*$. Moreover, for a generic $\alpha$ in \eqref{curve} the condition \eqref{Theta11} persists as well, and this proves the proposition. 
\end{proof}
\subsection{Proof of the main theorem} Having  \autoref{main} proved in the $k$-nondegenerate case (subsection 4.3) and having the relations \eqref{normalmap2},\eqref{p1term}, we are now in the position to prove \autoref{main} in its full generality.
\begin{proof}[Proof of \autoref{main}] 
According to the outcome of subsection 4.3, it remains to prove the theorem in the case when $M$ is $m$-nonminimal at $0$ but is not $k$-nondegenerate for any $k\geq 1$. Let us choose for $M,M^*$ local holomorphic coordinates according to \autoref{p1}. Then we have the identity \eqref{p1term}, for both the source and the target. Let us then consider the Segre family $\mathcal S=\{Q_p\}_{p=(\bar a,\bar b)}$ of $M$ as a $2$-parameter holomorphic family in $a,b$. Next, let us perform the following blow up in the space of parameters:
\begin{equation}\label{blowupab}
a=\tilde a\tilde b, \quad b=\tilde b.
\end{equation}
Let us denote the new parameterized family by $\widetilde{\mathcal S}$, and keep denoting for simplicity the new parameters by $a,b$. Then, considering an element of the family $\mathcal S$ as a graph $w=w(z)$ and expanding in $z,a,b$, we see that a terms $\lambda z^k a^j b^l$ gets transformed (after the blow up \eqref{blowupab}) into $\lambda z^k a^j b^{j+l}$. We obtain from here the crucial corollary that {\em all terms in the expansion of $w(z,a,b)$ except the very first term $z^0a^0b^1$ are divisible by $b^{p+1}$}. Furthermore, the non-zero (in view of \eqref{p1term}) term $\lambda z^kab^p,\,k\geq 1$ gets transformed into $z^kab^{p+1}$. We conclude that the transformed family $\widetilde{\mathcal S}$ has the form identical to \eqref{segrefam} with the nondegeneracy property \eqref{segredif}, with the only difference that $m$ is replaced by $p+1$. Hence, arguing identically to subsection 4.2, we conclude that the family $\widetilde{\mathcal S}$ (and hence the family $\mathcal S$!) satisfy a system of ODEs, identical to  \eqref{assocd} with the only difference that, again, $m$ is replaced by $p+1$. The same statement applies for the target $M^*$, and we conclude that the given formal map $(F,G)$ between $(M,0)$ and $(M^*,0)$ satisfies an identity similar to \eqref{tangency} with $m$ replaced by $p+1$. 

We finally recall that $(F,G)$ satisfies the factorization \eqref{normalmap2}, which is identical to \eqref{normalmap1} with, again, $m$ replaced by $p+1$. This allows to repeat the proof in the $k$-nondegenerate case word-by-word (we recall that the crucial \autoref{initterms} is valid without any further assumptions and thus is applicable to the map $(F,G)$). \autoref{main} is proved.

\end{proof}

\begin{proof}[Proof of \autoref{main2}]
The assertion of the theorem immediately follows from the crucial \autoref{summable} and \autoref{initterms}, and the representations \eqref{Finally},\eqref{finally1}.
\end{proof}




\end{document}